\documentclass{amsart}

\usepackage{amsmath}
\usepackage[colorlinks]{hyperref}
\usepackage{amsfonts}
\usepackage{MnSymbol}
\usepackage{amsthm}
\usepackage{newlfont}
\usepackage{amscd}
\usepackage{amsmath}
\usepackage{enumerate}
\usepackage[all,2cell]{xy}
\UseAllTwocells
\input xy
\xyoption{2cell}
\xyoption{all}
\usepackage{verbatim}
\usepackage{eucal}
\usepackage{xpatch}

\newcommand{\poc}[1][dr]{\save*!/#1+1.6pc/#1:(1,-1)@^{|-}\restore}

\newcommand{\Cof}{\mathcal{C}of}
\newcommand{\Fib}{\mathcal{F}ib}

\newcommand{\NN}{\mathbb{N}}

\newcommand{\LL}{\mathbb{L}}
\newcommand{\RR}{\mathbb{R}}
\newcommand{\EE}{\mathbb{E}}

\def\A{\mathcal{A}}

\def\C{\mathcal{C}}
\def\D{\mathcal{D}}
\def\G{\mathcal{G}}

\def\F{\mathcal{F}}
\def\SS{\mathcal{S}}
\def\M{\mathcal{M}}
\def\N{\mathcal{N}}
\def\W{\mathcal{W}}
\def\I{\mathcal{I}}
\def\cS{\mathcal{S}}
\def\T{\mathcal{T}}
\def\P{\mathcal{P}}

\def\J{\mathcal{J}}
\def\R{\mathcal{R}}
\renewcommand{\L}{\mathcal{L}}
\def\U{\mathcal{U}}

\def\W{\mathcal{W}}

\def\bS{\mathbf{S}}

\def\rH{\mathrm{H}}

\def\Om{{\Omega}}

\def\Del{{\Delta}}
\def\Sig{{\Sigma}}

\def\vphi{\varphi}

\newcommand{\HSwarrow}{\kern0.05ex\vcenter{\hbox{\Huge\ensuremath{\Swarrow}}}\kern0.05ex}
\newcommand{\hSwarrow}{\kern0.05ex\vcenter{\hbox{\huge\ensuremath{\Swarrow}}}\kern0.05ex}
\newcommand{\LLSwarrow}{\kern0.05ex\vcenter{\hbox{\LARGE\ensuremath{\Swarrow}}}\kern0.05ex}
\newcommand{\LSwarrow}{\kern0.05ex\vcenter{\hbox{\Large\ensuremath{\Swarrow}}}\kern0.05ex}

\newcommand{\HSearrow}{\kern0.05ex\vcenter{\hbox{\Huge\ensuremath{\Searrow}}}\kern0.05ex}
\newcommand{\hSearrow}{\kern0.05ex\vcenter{\hbox{\huge\ensuremath{\Searrow}}}\kern0.05ex}
\newcommand{\LLSearrow}{\kern0.05ex\vcenter{\hbox{\LARGE\ensuremath{\Searrow}}}\kern0.05ex}
\newcommand{\LSearrow}{\kern0.05ex\vcenter{\hbox{\Large\ensuremath{\Searrow}}}\kern0.05ex}

\newcommand{\HDownarrow}{\kern0.05ex\vcenter{\hbox{\Huge\ensuremath{\Downarrow}}}\kern0.05ex}
\newcommand{\hDownarrow}{\kern0.05ex\vcenter{\hbox{\huge\ensuremath{\Downarrow}}}\kern0.05ex}
\newcommand{\LLDownarrow}{\kern0.05ex\vcenter{\hbox{\LARGE\ensuremath{\Downarrow}}}\kern0.05ex}
\newcommand{\LDownarrow}{\kern0.05ex\vcenter{\hbox{\Large\ensuremath{\Downarrow}}}\kern0.05ex}

\newcommand{\HUparrow}{\kern0.05ex\vcenter{\hbox{\Huge\ensuremath{\Uparrow}}}\kern0.05ex}
\newcommand{\hUparrow}{\kern0.05ex\vcenter{\hbox{\huge\ensuremath{\Uparrow}}}\kern0.05ex}
\newcommand{\LLUparrow}{\kern0.05ex\vcenter{\hbox{\LARGE\ensuremath{\Uparrow}}}\kern0.05ex}
\newcommand{\LUparrow}{\kern0.05ex\vcenter{\hbox{\Large\ensuremath{\Uparrow}}}\kern0.05ex}

\newtheorem{thm}{Theorem}[subsection]
\newtheorem{cor}[thm]{Corollary}
\newtheorem{lem}[thm]{Lemma}
\newtheorem{pro}[thm]{Proposition}

\makeatletter
\@addtoreset{thm}{section}
\makeatother

\numberwithin{equation}{subsection}
\numberwithin{thm}{subsection}

\theoremstyle{definition}
\newtheorem{define}[thm]{Definition}

\newtheorem{exam}[thm]{Example}
\newtheorem{examples}[thm]{Examples}
\newtheorem{defn}[thm]{Definition}

\theoremstyle{remark}
\newtheorem{rem}[thm]{Remark}

\DeclareMathOperator{\hocolim}{hocolim}
\DeclareMathOperator{\colim}{colim}

\DeclareMathOperator{\Id}{Id}
\DeclareMathOperator{\id}{id}
\DeclareFontFamily{OT1}{pzc}{}
\DeclareFontShape{OT1}{pzc}{m}{it}{<-> s * [1.10] pzcmi7t}{}
\DeclareMathAlphabet{\mathpzc}{OT1}{pzc}{m}{it}
\DeclareMathOperator{\cof}{cof}
\DeclareMathOperator{\fib}{fib}

\DeclareMathOperator{\Alg}{Alg}

\DeclareMathOperator{\ModCat}{ModCat}
\DeclareMathOperator{\sGr}{sGr}

\DeclareMathOperator{\op}{op}
\DeclareMathOperator{\Map}{Map}

\DeclareMathOperator{\df}{def}
\DeclareMathOperator{\Set}{Set}

\DeclareMathOperator{\Cat}{Cat}

\DeclareMathOperator{\AdjCat}{AdjCat}

\DeclareMathOperator{\Ob}{Ob}
\DeclareMathOperator{\Sp}{Sp}
\DeclareMathOperator{\Ho}{Ho}

\DeclareMathOperator{\ev}{ev}
\DeclareMathOperator{\fin}{fin}

\DeclareMathOperator{\dg}{dg}

\DeclareMathOperator{\Tw}{Tw}

\DeclareMathOperator{\Joy}{Joy}

\DeclareMathOperator{\Spectra}{Spectra}

\DeclareMathOperator{\Reedy}{Reedy}

\DeclareMathOperator{\Exc}{Exc}

\DeclareMathOperator{\codom}{codom}

\DeclareMathOperator{\hQ}{Q}

\def\x{\overset}

\def\Hom{\textrm{Hom}}

\newcommand{\tgpd}{\kern0.05ex\vcenter{\hbox{\footnotesize\ensuremath{2}}}\kern0.05ex\mathcal{G}pd} 

\def\rar{\rightarrow}

\def\lrar{\longrightarrow}

\def\hrar{\hookrightarrow}

\newcommand{\adj}{\mathrel{\substack{\longrightarrow \\[-.6ex] \x{\upvdash}{\longleftarrow}}}}

\def\ovl{\overline}

\def\bar{\overline}

\title{The tangent bundle of a model category}

\author{Yonatan Harpaz}
\email{harpaz@math.univ-paris13.fr}
\address{Institut Galil\'ee\\ Universit\'e Paris 13\\ 99 avenue J.B. Cl\'ement\\ 93430 Villetaneuse\\ France.}
\author{Joost Nuiten}
\email{j.j.nuiten@uu.nl}
\address{Mathematical Institute\\ Utrecht University\\ P.O. Box 80010\\ 3508 TA Utrecht\\ The
Netherlands.}
\author{Matan Prasma}
\email{mtnprsm@gmail.com}
\address{Faculty of Mathematics\\ University of Regensburg\\ Universitatsstrase 31, 93040\\ Germany.}

\date{}

\begin{document}

\begin{abstract}
This paper studies the homotopy theory of parameterized spectrum objects in a model category from a global point of view. More precisely, for a model category $\M$ satisfying suitable conditions, we construct a relative model category $\T\M\lrar \M$, called the tangent bundle, whose fibers are models for spectra in the various over-categories of $\M$, and which presents the $\infty$-categorical tangent bundle.
Moreover, the tangent bundle $\T\M$ inherits an enriched model structure when such a structure exists on $\M$. This additional structure is used in subsequent work to identify the tangent bundles of algebras over an operad and of enriched categories.
\end{abstract}

\maketitle

\tableofcontents

\section{Introduction}

This paper is part of an on going work concerning the abstract cotangent complex and Quillen cohomology. In \cite{HNP17a} and \cite{HNP17b}, the authors study the tangent categories of algebras over an operad and of enriched categories, as well as their cotangent complex. The work in \textit{loc. cit.} was carried out in a model-categorical framework. The subject of this note, which was splitted from \cite{HNP17a}, is the development of such a framework, which we believe is of independent interest. 

The theory of the (spectral) cotangent complex has been developed in the works of~\cite{Sch97}, \cite{BM05} and most recently in~\cite{Lur14} in the setting of $\infty$-categories: associated to an $\infty$-category $\C$ and an object $A$ in $\C$ is an $\infty$-category $\Sp(\C_{/A})$ of \textbf{spectrum objects} in $\C_{/A}$. The $\infty$-category $\Sp(\C_{/A})$ can be expressed as the $\infty$-category of \textbf{reduced excisive functors} from pointed finite spaces to $\C_{/A}$ or in other words, \textbf{linear functors} from finite pointed spaces to $\C_{/A}$ in the sense of Goodwillie (\cite{Goo91}). Following \cite{Lur14}, we will denote $\T_{A}\C:=\Sp(\C_{/A})$ and refer to it as the \textbf{tangent $\infty$-category of $\C$ at $A$}. 

The tangent $\infty$-category $\T_A\C$ can be considered as a homotopy theoretical analogue of the category of abelian group objects over $A$, which are classically known as \textbf{Beck modules} (\cite{Bec67}).  
In particular, when $\C$ is a presentable $\infty$-category, there is a natural ``linearization" functor $$\Sigma^\infty_+:\C_{/A}\lrar \T_A\C, $$ 
which leads to the notion of the \textbf{cotangent complex} $L_{A}:=\Sigma^{\infty}_+(\id:A\lrar A)\in \T_A\C$. As in the classical case (\cite{Qui67}), the Quillen cohomology groups of $A$ with coefficients in $E\in\T_A\C$ are then given by the formula $\rH^n_{\hQ}(A,E):=\pi_0\Map_{\T_A\C}(L_A,E[n])$. 

In order to view the cotangent complex $L_A$ mentioned above as a functor in $A$ we consider the \textbf{tangent bundle of $\C$}, $\T\C:=\int_A\Sp(\C_{/A}) \lrar \C$ and define the cotangent complex functor $L:\C\lrar \T\C$ as the composite $$\C\overset{\Delta}{\lrar} \C^{\Delta^1}=\int_A\C_{/A}\lrar \int_A\T_A\C=\T\C.$$ Once such a view-point is set, many of the properties of Quillen cohomology become a consequence of corresponding properties of the cotangent complex. One such example is that Quillen cohomology always admits a \textbf{transitivity sequence}, as in the classical case of Andr\'{e}-Quillen cohomology (\cite{Qui70}). 
          
The purpose of this paper is to develop model-categorical tools to study the cotangent complex formalism in the vein of \cite[\S 7.3]{Lur14}. 
We will start by describing a model $\Sp(\M)$ for the stabilization of a model category $\M$, which does not require the loop-suspension adjunction to arise from a Quillen pair. This is useful, for example, for model categories of enriched categories, or enriched operads, which do not offer natural choices for such a Quillen adjunction (see \cite{HNP17b}). We then describe how the usual machinery of suspension- and $\Omega$-spectrum replacements arises in our setting. When applied to pointed objects in $\M_{/A}$, the model above gives the \textbf{tangent model category} $\T_A\M$ at $A$. 

In the second half of the paper we use a similar approach to construct a model $\pi:\T\M \lrar \M$ for the \textbf{tangent bundle} of $\M$, namely, a presentation of the $\infty$-categorical projection
$$
\int_{A\in \M_\infty}\T_A\M_\infty\lrar \M_\infty
$$
whose fibers are the tangent $\infty$-categories of the $\infty$-category $\M_\infty$ underlying $\M$. Our main results are that $\T\M$ enjoys particularly favorable properties on the model categorical level: it exhibits $\T\M$ as a \textbf{relative model category} over $\M$, in the sense of~\cite{HP}, and forms a \textbf{model fibration} when restricted to the fibrant objects of $\M$. Furthermore, when $\M$ is tensored over a suitable model category $\bS$, the tangent bundle $\T\M$ inherits this structure, and thus becomes enriched in $\bS$. This enrichment plays a key role in the description of the tangent categories of algebras and enriched categories as in \cite{HNP17a} and\cite{HNP17b}, and may be useful for other purposes as well.    

  
 
\section{Tangent model categories}
In this section we discuss a particular model-categorical presentation for the homotopy theory of spectra in a -- sufficiently nice -- model category $\M$, as well as a model $\T\M$ for the homotopy theory of spectra parameterized by the various objects of $\M$. The model category $\Sp(\M)$ of spectrum objects in $\M$ presents the universal stable $\infty$-category associated to the $\infty$-category underlying $\M$. When $\M$ is a simplicial model category, one can use the suspension and loop functors induced by the simplicial (co)tensoring to give explicit models for spectrum objects in $\M$ by means of Bousfield-Friedlander spectra or symmetric spectra (see \cite{Hov}). In non-simplicial contexts this can be done as soon as one chooses a Quillen adjunction realizing the loop-suspension adjunction.

The main purpose of this section is to give a uniform description of stabilization which does not depend on a simplicial structure or any other specific model for the loop-suspension adjunction. We will consequently follow a variant of the approach suggested by Heller in~\cite{Hel}, and describe spectrum objects in terms of $(\NN\times \NN)$-diagrams (see also~\cite[\S 8]{Lur06}). This has the additional advantage of admitting a straightforward `global' analogue $\T\M$, which will focus on in \S\ref{s:tanbun} and \S\ref{s:tanten}.

\subsection{Spectrum objects}\label{s:stabilization}
Suppose that $\M$ is a weakly pointed model category, i.e.\ the homotopy category of $\M$ admits a zero object. If $X \in \M$ is a cofibrant object and $Y \in \M$ is a fibrant object, then a commuting square 
\begin{equation}\label{e:sqr}\vcenter{\xymatrix{
X\ar[r]\ar[d] & Z\ar[d]\\
Z'\ar[r] & Y
}}
\end{equation}
in which the objects $Z$ and $Z'$ are weak zero objects is equivalent to the datum of a map $\Sigma X \lrar Y$, or equivalently, a map $X\lrar \Omega Y$.
The square~\eqref{e:sqr} is homotopy coCartesian if and only if the corresponding map $\Sig X \lrar Y$ is an equivalence, and is homotopy Cartesian if and only if the adjoint map $X \lrar \Om Y$ is an equivalence. 

Using this, one can describe (pre-)spectra in terms of $(\NN\times\NN)$-diagrams
$$\xymatrix@C=1.3pc@R=1.3pc{X_{00}\ar[r]\ar[d] & X_{01} \ar[r]\ar[d] & \cdots\\ X_{10}\ar[r]\ar[d] & X_{11}\ar[r]\ar[d] & \cdots \\
\vdots & \vdots &}$$
in which all the off-diagonal entries are weak zero objects. Indeed, the diagonal squares
\begin{equation}\label{e:diag}\vcenter{\xymatrix{
X_{n, n}\ar[r]\ar[d] & X_{n, n+1}\ar[d]\\
X_{n+1, n}\ar[r] & X_{n+1, n+1}
}}\end{equation}
describe the structure maps of the pre-spectrum. 
\begin{define}\label{d:spectra}
Let $\M$ be a weakly pointed model category. We will say that an $(\NN\times\NN)$-diagram $X_{\bullet,\bullet}: \NN\times\NN\lrar \M$ is
\begin{enumerate}[(1)]
\item 
a \textbf{pre-spectrum} if all its off-diagonal entries are weak zero objects in $\M$;
\item 
an \textbf{$\Om$-spectrum} if it is a pre-spectrum and for each $n\geq 0$, the diagonal square \eqref{e:diag} is homotopy Cartesian;
\item 
a \textbf{suspension spectrum} if it is a pre-spectrum and for each $n\geq 0$, the diagonal square \eqref{e:diag} is homotopy coCartesian.
\end{enumerate}
\end{define}
The category $\NN\times \NN$ has the structure of a \textbf{Reedy category} (see~\cite[Definition 5.2.1]{Hov99}) in which all maps are increasing. It follows that $\M^{\NN\times\NN}$ carries the Reedy model structure, which agrees with the projective model structure.

\begin{define}\label{d:stable-equiv}
Let $\M$ be a weakly pointed model category. We will say that a map $f: X \lrar Y$ in $\M^{\NN \times \NN}$ is a \textbf{stable equivalence} if for every $\Om$-spectrum $Z$ the induced map on derived mapping spaces
$$ \Map^h(Y,Z) \lrar \Map^h(X,Z) $$
is a weak equivalence. A stable equivalence between $\Om$-spectra is always a levelwise equivalence.
\end{define}

\begin{defn}\label{d: stabilization model category}
Let $\M$ be a weakly pointed model category. The \textbf{stable model structure} on the category $\M^{\NN\times\NN}$ is -- if it exists -- the model structure whose 
\begin{itemize}
\item cofibrations are the Reedy cofibrations.
\item weak equivalences are the stable equivalences.
\end{itemize}
When it exists, we will denote this model category by $\Sp(\M)$ and refer to it as the \textbf{stabilization} of $\M$.
\end{defn}

To place the terminology of Definition~\ref{d: stabilization model category} in context, recall that a model category $\M$ is called \textbf{stable} if it is weakly pointed and $\Sigma:\Ho(\M)\adj \Ho(\M):\Omega$ is an equivalence of categories (cf.~\cite{Hov}). Equivalently, $\M$ is stable if the underlying $\infty$-category $\M_{\infty}$ (see Section \ref{s:oo-stab}) is stable in the sense of~\cite[\S 1]{Lur14}, i.e.~if $\M_\infty$ is pointed and the adjunction of $\infty$-categories $\Sigma: \M_\infty \adj \M_\infty : \Om$ is an adjoint equivalence. This follows immediately from the fact that an adjunction between $\infty$-categories is an equivalence if and only if the induced adjunction on homotopy categories is an equivalence.

\begin{rem}\label{r:lurie}
Alternatively, one can characterize the stable model categories as those weakly pointed model categories in which a square 
is homotopy Cartesian if and only if it is homotopy coCartesian (see~\cite[\S 1]{Lur14}).
\end{rem}

\begin{pro}\label{l:stability}
Let $\M$ be a weakly pointed model category. Then $\Sp(\M)$ is -- if it exists -- a stable model category.
\end{pro}
\begin{proof}
Observe that $\Sp(\M)$ comes equipped with an adjoint pair of \textbf{shift functors}
$$\xymatrix{[-n]:\Sp(\M)\ar[r]<1ex> & \Sp(\M):[n]\ar[l]<1ex> & n\geq 0}$$
given by $X[n]_{\bullet\bullet}:=X_{\bullet+n,\bullet+n}$ and $X[-n]_{\bullet,\bullet}=X_{\bullet-n,\bullet-n}$. Here $X_{i,j}=\emptyset$ when $i<0$ or $j<0$. These form a Quillen pair
since the functor $[-n]$ preserves levelwise weak equivalences and cofibrations, while $[n]$ preserves $\Omega$-spectra. For each $\Om$-spectrum $Z$, there is a natural isomorphism $Z\lrar \Om(Z[1])$ in $\Ho(\Sp(\M))$, which shows that $\Om\circ \RR[1]$ is equivalent to the identity. On the other hand, $[1]$ is a right Quillen functor so that $\Om\circ \RR[1]\simeq \RR[1]\circ \Om$, which shows that $\Om: \Ho(\Sp(\M))\lrar \Ho(\Sp(\M))$ is an equivalence.
\end{proof}

When $\M$ is left proper, the fibrant objects of $\Sp(\M)$ are precisely the Reedy fibrant $\Om$-spectra in $\M$ (see \cite[Proposition 3.4.1]{Hir}). 
On the other hand, $\Om$-spectra can always be characterized as the local object against a particular class of maps:
\begin{lem}\label{l:localizingmaps}
Let $\M$ be a weakly pointed model category and let $\G$ be a class of cofibrant objects in $\M$ with the following property: a map $f: X \lrar Y$ in $\M$ is a weak equivalence if and only if the induced map 
$$
\Map^h_\M(D,X) \lrar \Map^h_{\M}(D,Y)
$$ 
is a weak equivalence of spaces for every $D \in \G$. Then an object $Z\in \M^{\NN\times\NN}$ is:
\begin{enumerate}[(1)]
\item  a pre-spectrum if and only if $Z$ is local with respect to the set of maps
$$ (\star)\quad \emptyset \lrar h_{n,m} \otimes D $$
for every $D \in \G$ and $n \neq m$, where $h_{n, m}=\hom((n,m),-):\NN\times \NN\lrar \Set$ and $\otimes$ denotes the natural tensoring of $\M$ over sets.
\item 
an $\Omega$-spectrum if and only if it is a pre-spectrum which is furthermore local with respect to the set of maps
$$ (\star\star)\quad \left[h_{n+1, n}\coprod_{h_{n+1, n+1}} h_{n, n+1}\right] \otimes D \lrar h_{n,n} \otimes D $$
for every $D \in \G$ and every $n \geq 0$.
\end{enumerate}
\end{lem}
\begin{proof}
Let $Z$ be a Reedy fibrant object of $\M^{\NN\times\NN}$. For any object $A\in \M$, the diagram $h_{n, m}\otimes A$ is the image of $A$ under the left adjoint to the functor $\M^{\NN \times \NN} \lrar \M; Z\mapsto Z_{n, m}$. Unwinding the definitions, the image of ($\star$) under $\Map^h(-, Z)$ can therefore be identified with the map
$$
\Map_\M^h(D, Z_{n, m})\lrar \Map_\M^h(\emptyset, Z_{n, m})\simeq \ast.
$$
It follows that $Z$ is local with respect to the maps ($\star$) iff $Z_{n, m}$ is a weak zero object, which proves (1).

For (2), observe that for any cofibrant object $D$ and any pair of $n'\leq n$, $m'\leq m$, the maps $h_{n, m}\otimes D\lrar h_{n', m'}\otimes D$ are levelwise cofibrations between Reedy cofibrant objects. Since homotopy pushouts in $\M^{\NN\times\NN}$ are computed levelwise, it follows that the domain of ($\star\star$) is a homotopy pushout of $\NN\times\NN$-diagrams. Using this, the image of ($\star\star$) under $\Map^h(-, Z)$ can therefore be identified with the map
$$
\Map^h_\M(D, Z_{n, n})\lrar \Map^h_{\M}(D, Z_{n+1, n})\times^h_{\Map^h_\M(D, Z_{n+1, n+1})} \Map^h_\M(D, Z_{n, n+1}).
$$
The target of this map can be identified with $\Map^h_\M\big(D, Z_{n,n+1} \times^h_{Z_{n+1,n+1}} Z_{n+1,n}\big)$. It follows that a pre-spectrum is local with respect to ($\star\star$) iff it is an $\Omega$-spectrum.
\end{proof}
\begin{cor}\label{c:existence of Sp}
Let $\M$ be a left proper combinatorial model category which is weakly pointed. Then the stabilization $\Sp(\M)$ exists. 
\end{cor}
\begin{proof}
Because $\M$ is combinatorial there exists a \emph{set} $\G$ of cofibrant objects of $\M$ which together detect weak equivalences as above (see e.g.~\cite[Proposition 4.7]{Dug}). The stable model structure can therefore be identified with the left Bousfield localization of the Reedy model structure at a set of maps, which exists because $\M$ is left proper (see~\cite[Theorem 4.1.1]{Hir}).
\end{proof}

\begin{pro}\label{p:induced-adj-2}
If $\L:\M \adj \N:\R$ is a Quillen adjunction between left proper combinatorial model categories then its levelwise prolongation
$$\xymatrix@C=4pc{
\Sp(\M)\ar@<1ex>[r]^{\Sp(\L)} & \Sp(\N)\ar@<1ex>[l]^{\Sp(\R)}_{\upvdash}\\
}$$
is a Quillen adjunction with respect to the stable model structures on both sides. Furthermore, if $\L \dashv\R$ is a Quillen equivalence then so is $\Sp(\L) \dashv\Sp(\R)$.
\end{pro}
\begin{proof}
Since $\RR\R^{\NN\times\NN}: \N^{\NN \times \NN} \lrar \M^{\NN \times \NN}$ preserves $\Omega$-spectra it follows from~\cite[Theorem 3.1.6, Proposition 3.3.18]{Hir} that the Quillen adjunction $\R^{\NN \times \NN} \dashv \L^{\NN \times \NN}$ descends to the stable model structure. A Quillen equivalence $\L\dashv \R$ induces a Quillen equivalence between Reedy model structures. This implies that the induced Quillen pair between stabilizations is a Quillen equivalence as well. Indeed, the right Quillen functor $\Sp(\R)$ detects equivalences between Reedy fibrant $\Om$-spectra (which are just levelwise equivalences) and the derived unit map of $\Sp(\R)$ can be identified with the derived unit map of $\R^{\NN\times \NN}$.
\end{proof}

\begin{rem}\label{r:replacement}
When $\M$ is combinatorial and weakly pointed, any Reedy cofibrant object $X\in \M^{\NN\times\NN}$ admits a stable equivalence $X\lrar E$ to an $\Omega$-spectrum. This either follows formally from inspecting the proof of the existence of Bousfield localizations in the left proper case, or -- if $\M$ is differentiable -- from the explicit constructions in Remark \ref{r:prespectrumrep} and Corollary \ref{c:fibrantreplacement}. 
\end{rem}

\begin{rem}\label{r:stablewe}
When the stable model structure does not exist, the class of Reedy cofibrations which are also stable equivalences is not closed under pushouts. However, this class is closed under pushouts along maps with a levelwise cofibrant domains and codomains (indeed, such pushouts are always homotopy pushouts in the injective model structure on $\M^{\NN \times \NN}$ and hence in the Reedy model structure as well).
\end{rem}

\subsection{Parameterized spectrum objects}\label{s:tanbun}
In the previous section we have seen that any -- sufficiently nice -- weakly pointed model category $\M$ gives rise to a model category $\Sp(\M)$ of spectra in $\M$, depending naturally on $\M$. One can mimic the description of $\Sp(\M)$ in terms of $\NN\times\NN$-diagrams to produce a model category $\T\M$ of \textbf{parameterized spectra} in a model category $\M$, with varying `base spaces'. Indeed, for a fixed base $A \in M$, consider the pointed model category $\M_{A//A}$ of retractive objects over $A$, i.e.\ maps $B\lrar A$ equipped with a section. An $\Om$-spectrum in $\M_{A//A}$ can be considered as a \textbf{parameterized spectrum} over $A$. This notion was first studied by May and Sigurdsson in~\cite{MS06} when $\M$ is the category of topological spaces. In that case a paramterized spectrum over $A$ describes a functor from the fundametal $\infty$-groupoid of $A$ to spectra (see \cite[Appendix B]{MBG11}). When $\M$ is the category of $\EE_{\infty}$-ring spectra, Basterra and Mandell (\cite{BM05}) showed that parameterized $\Om$-spectra over $R \in \M$ is essentially equivalent to the notion of an $E$-module spectrum.


\begin{define}
Let $\M$ be a model category. We will denote by
$$ \T_A\M := \Sp(\M_{A//A}) $$
the stabilization of $\M_{A//A}$, when it exists, and refer to it as the \textbf{tangent model category} of $\M$ at $A$.
\end{define}

\begin{rem}
When $\M$ is combinatorial and left proper then $\M_{A//A}$ is combinatorial and left proper for every $A$ and so all the tangent $\T_\A\M$ exists for every $A$. In \S\ref{s:oo-stab} we will show that under mild conditions the model category $\T_A\M$ is also presentation of the tangent $\infty$-category $\T_A\M_{\infty}$.
\end{rem}


Note that a spectrum in $\M_{A//A}$ is given by the datum of a diagram $X: \NN\times \NN\lrar \M_{A//A}$, which is equivalent to the datum of a diagram
$$
X': (\NN\times \NN)_\ast \lrar \M
$$
such that $X'(\ast)=A$, where $(\NN\times\NN)_\ast$ denotes the category obtained from $\NN\times\NN$ by \textbf{freely adding a zero object} $\ast$. In other words, for a category $\I$, the category $\I_\ast$ has object set $\Ob(\I) \cup \{\ast\}$, and maps $\Hom_{\I_\ast}(i,j) = \Hom_{\I}(i,j) \cup \{\ast\}$ for every $i,j \in \I$, and $\Hom_{\I_\ast}(i,\ast) = \Hom_{\I_\ast}(\ast,i) = \{\ast\}$ for every $i \in \I$ (here the composition of $\ast$ with any other map is again $\ast$). Parameterized spectra with varying base can therefore be described in terms of $(\NN \times \NN)_\ast$ diagrams whose value at $\ast$ is not fixed in advance.

\begin{rem}
When $\I=(\I, \I^+, \I^-)$ is a Reedy category $\I_\ast$ is again a Reedy category, where we consider $\ast \in \I_\ast$ as being the unique object of degree $0$ and such that for every $i$ the unique map $\ast \lrar i$ is in $\I_\ast^+=(\I_\ast)^+$ and the unique map $i \lrar \ast$ is in $\I_\ast^-$. 
\end{rem}

\begin{define}\label{d:the-other-way}
Let $\M$ be a model category and let $X: (\NN\times \NN)_\ast\lrar \M$ be a diagram. We will say that $X$ is a \textbf{parameterized $\Omega$-spectrum} in $\M$ if it is satisfies the following two conditions:
\begin{enumerate}
\item for each $n\neq m$, the map $X(n, m)\lrar X(\ast)$ is a weak equivalence.
\item for each $n \geq 0$ the square 
\begin{equation}\label{e:square_nn}\vcenter{
\xymatrix{
X_{n,n} \ar[r]\ar[d] & X_{n+1,n} \ar[d] \\
X_{n+1,n} \ar[r] & X_{n+1,n+1} \\
}}
\end{equation}
is homotopy Cartesian.
\end{enumerate}
We will say that a map $f: X \lrar Y$ in $\M^{(\NN \times \NN)_\ast}$ is a \textbf{stable equivalence} if for every parameterized $\Om$-spectrum $Z$ the induced map on derived mapping spaces
$$ \Map^h(Y,Z) \lrar \Map^h(X,Z) $$
is an equivalence.
\end{define}

\begin{rem}\label{r:omega}
A diagram $X: (\NN\times \NN)_\ast\lrar \M$ is a Reedy fibrant parameterized $\Om$-spectrum iff $X(\ast)$ is fibrant in $\M$ and $X$ determines a Reedy fibrant $\Om$-spectrum in $\M_{X(\ast)//X(\ast)}$.
\end{rem}

\begin{define}
The \textbf{tangent bundle} $\T\M$ of $\M$ is -- if it exists -- the unique model structure on $\M^{(\NN \times \NN)_\ast}$ whose cofibrations are the Reedy cofibrations and whose weak equivalences are the stable equivalences.
\end{define}

When the tangent bundle $\T\M$ exists it has the same cofibrations and less fibrant objects than the Reedy model structure. It follows that $\T\M$ is a left Bousfield localization of the Reedy model structure. In fact, Lemma \ref{l:localizingmaps} shows that $\T\M$ can be obtained from the Reedy model structure by left Bousfield localizing at the class of maps
$$ h_{\ast} \otimes D \lrar h_{n,m} \otimes D \quad \quad n \neq m, D \in \G $$
together with the maps
$$  \left[h_{n+1, n} \coprod_{h_{n+1, n+1}} h_{n, n+1}\right] \otimes D \lrar h_{n, n} \otimes D \quad\quad n \geq 0, D \in \G. $$
Here $h_x: (\NN\times \NN)_\ast\lrar \Set$ is the functor corepresented by $x\in (\NN\times\NN)_\ast$, $\otimes$ denotes the natural tensoring of $\M$ over sets and $\G$ is a class of cofibrant objects $D$ such that the functors $\Map^h_{\M}(D, -)$ mutually detect equivalences. 

Corollary \ref{c:existence of Sp} has the following analogues:
\begin{cor}
If $\M$ is a left proper combinatorial model category, then the tangent bundle $\T\M$ exists.
\end{cor}

\begin{examples}\
\begin{enumerate}
\item
When $\M=\cS$ is the category of simplicial sets with the Kan-Quillen model structure then $\T_X\cS$ gives a model for parameterized spectra over $X$ which is equivalent to that of~\cite{MS06} (see~\cite[\S 2.3]{HNP17b}). Similarly, $\T\cS$ is the associated global model, whose objects can be thought of as pairs consisting of a space $X$ together with a parameterized spectrum over $X$.
\item
When $\M=\sGr$ is the category of simplicial groups the tangent model category $\T_{G}\sGr$ is Quillen equivalent to the model category of (naive) $G$-spectra (see \cite[\S 2.4]{HNP17b}).
\item
If $\P$ is a cofibrant dg-operad over a field $k$ of characteristic $0$ and $\M=\dg\Alg_\P$ is the model category of dg-algebras over $\P$ then $\M$ is left proper and for every $\P$-dg-algebra object $A$ the tangent category $\T_\A\M$ is Quillen equivalent to the category of dg-$A$-modules (see \cite{Sch97}, \cite{HNP17a}).
\item
If $\bS$ is an excellent symmetric monoidal model category in the sense of~\cite[\S A.3]{Lur09} and $\M =\Cat_{\bS}$ is the model category of small $\bS$-enriched categories then for every fibrant $\bS$-enriched category $\C$ the tangent category $\T_{\C}\Cat_{\bS}$ is Quillen equivalent to the category of enriched lifts $\C^{\op} \otimes \C \lrar \T\bS$ of the mapping space functor $\Map: \C^{\op} \otimes \C \lrar \bS$ (\cite[Corollary 3.1.16]{HNP17b}). 
\item
If $\M = \Set_{\Del}^{\Joy}$ is the model category of simplicial sets endowed with the \textbf{Joyal} model structure and $\C \in \Set_{\Del}^{\Joy}$ is a fibrant object (i.e., an $\infty$-category) then $\T_{\C}\Set_{\Del}^{\Joy}$ is equivalent to the model category $(\Set_{\Del})_{/\Tw(\C)}$ of simplicial sets over the \textbf{twisted arrow category} of $\C$ equipped with the covariant model structure. In particular, the underlying $\infty$-category $\T_{\C}\Cat_{\infty} \simeq \left(\T_{\C}\Set_{\Del}^{\Joy}\right)_{\infty}$ is equivalent to the $\infty$-category of functors $\Tw(\C) \lrar \Spectra$ (see \cite[Corollary 3.3.1]{HNP17b}).
\end{enumerate}
\end{examples}

\subsection{Suspension spectra}

In section \S\ref{s:stabilization} we considered a model for spectrum objects in a weakly pointed model category $\M$, and saw that in good cases it yields a model category $\Sp(\M)$. We will now show that when this holds, one can also model the classical ``suspension-infinity/loop-infinity'' adjunction via a Quillen adjunction.

\begin{pro}\label{p:induced-adj}
Let $\M$ be a weakly pointed model category such that $\Sp(\M)$ exists. Then the adjunction
$$ \Sig^\infty : \M \adj \Sp(\M) : \Om^\infty $$
given by $\Sig^{\infty}(X)_{n,m} = X$ and $\Om^{\infty}(X_{\bullet\bullet}) = X_{0,0}$ is a Quillen adjunction. Furthermore, this Quillen adjunction is natural in $\M$ in the following sense: for any Quillen pair 
$\L:\M \adj \N:\R$ between two such model categories the diagram of Quillen adjunctions 
$$\xymatrix@C=4pc{
\Sp(\M)\ar@<1ex>[r]^{\Sp(\L)}\ar@<1ex>[d]^{\Omega^\infty}_-\dashv & \Sp(\N)\ar@<1ex>[l]^{\Sp(\R)}_{\upvdash}\ar@<1ex>[d]^{\Omega^\infty}_-\dashv\\
\M\ar@<1ex>[u]^{\Sigma^\infty}\ar@<1ex>[r]^{\L} & \N\ar@<1ex>[l]^{\R}_{\upvdash}\ar@<1ex>[u]^{\Sigma^\infty}
}$$ 
commutes.
\end{pro}
\begin{proof}
The functor $\ev_{(0, 0)}\colon \M^{\NN\times\NN}\lrar \M$ evaluating at $(0, 0)$ is already right Quillen functor for the Reedy model structure, so in particular for the stable model structure on $\M^{\NN\times \NN}$. The commutation of the diagram is immediate to check on right adjoints. 

\end{proof}

The adjunction $\Sig^{\infty}: \M \adj \Sp(\M):\Om^{\infty}$ of Proposition~\ref{c:existence of Sp} is offered as a model for the classical suspension-infinity/loop-infinity adjunction. This might seem surprising at first sight as the object $\Sig^{\infty}(X)$ is by definition a \textbf{constant $(\NN \times \NN)$-diagram}, and not a suspension spectrum. In this section we will prove a convenient replacement lemma showing that up to a stable equivalence every constant spectrum object can be replaced with a suspension spectrum, which is unique in a suitable sense (see Remark~\ref{r:unique}). This can be used, for example, in order to functorially replace $\Sig^{\infty}(X)$ with a suspension spectrum, whenever the need arises (see Corollary~\ref{c:suspensionspectrum} below). While mostly serving for intuition purposes in this paper, Lemma~\ref{l:suspensionspectrum} is also designed for a more direct application in~\cite{HNP17b}. 

\begin{lem}\label{l:suspensionspectrum}
Let $\M$ be a combinatorial model category. Let $f:X \lrar Y$ be a map in $\M^{\NN \times \NN}$ such that $X$ is constant and levelwise cofibrant and $Y$ is a suspension spectrum. Then there exists a factorization $X \x{f'}{\lrar} X' \x{f''}{\lrar} Y$ of $f$ such that $X'$ is a suspension spectrum, $f'$ is a stable equivalence and $f'_{0,0}: X_{0,0} \lrar X'_{0,0}$ is a weak equivalence. In particular, if $f_{0,0}: X_{0,0} \lrar Y_{0,0}$ is already a weak equivalence then $f$ is a stable equivalence.
\end{lem}
\begin{proof}
Let us say that an object $Z_{\bullet\bullet} \in \Sp(\M)$ is a suspension spectrum up to $n$ if $Z_{m,k}$ is weak zero objects whenever $m \neq k$ and $\min(m,k)<n$ and if the $m$'th diagonal square is a pushout square for $m < n$. In particular, the condition of being a suspension spectrum up to $0$ is vacuous. We will now construct a sequence of levelwise cofibrations and stable equivalences 
$$ X = P_0 \lrar P_1 \lrar \cdots \lrar P_{n} \lrar P_{n+1} \lrar \cdots $$
over $Y$ such that each $P_n$ is a levelwise cofibrant suspension spectrum up to $n$ and the map $(P_n)_{m,k} \lrar (P_{n+1})_{m,k}$ is an isomorphism whenever $\min(m,k)< n$ or $m=k=n$. Then $X' \x{\df}{=} \colim_n P_n \simeq \hocolim_n P_n$ is a suspension spectrum by construction and the map $f: X \lrar X'$ satisfies the required conditions (see Remark~\ref{r:stablewe}). 

Given a cofibrant object $Z \in \M$ equipped with a map $Z \lrar Y_{n,n}$, let us denote the cone of the composed map $Z \lrar Y_{n,n} \lrar Y_{n,n+1}$ by $Z \lrar C_{n,n+1}(Z) \lrar Y_{n,n+1}$ and the cone of the map $Z \lrar Y_{n,n} \lrar Y_{n+1,n}$ by $Z \lrar C_{n+1,n}(Z) \lrar Y_{n+1,n}$. 
Since $Y$ is weakly contractible off diagonal it follows that $C_{n,n+1}(Z)$ and $C_{n+1,n}(Z)$ are weak zero objects. Let $\Sigma_Y(Z):= C_{n,n+1}(Z)\coprod_Z C_{n+1,n}(Z)$ be the induced model for the suspension of $Z$ in $\M$. By construction the object $\Sig_Y(Z)$ carries a natural map $\Sig_Y(Z) \lrar Y_{n+1,n+1}$. Let us now define $Q_{n,n+1}(Z)$, $Q_{n+1,n}(Z)$ and $Q_{n+1}(Z)$ by forming the following diagram in $\M^{\NN \times \NN}_{/Y}$:
$$
\xymatrix{ 
& h_{n,n+1}\otimes Z \ar[d]\ar[r] & h_{n,n+1}\otimes C_{n,n+1}(Z)\ar[d] \\ 
h_{n+1,n} \otimes Z \ar[r]\ar[d] & h_{n, n}\otimes Z\ar[r]\ar[d] & \poc Q_{n,n+1}(Z) \ar[d] &\\
h_{n+1,n} \otimes C_{n+1,n}(Z) \ar[r] & \poc Q_{n+1,n}(Z) \ar[r] & \poc Q_{n+1}(Z). \\
}$$
Since all objects in this diagram are levelwisewise cofibrant and the top right horizontal map is a levelwise cofibration and a stable equivalence, all the right horizontal maps are levewise cofibrations and stable equivalences (see Remark~\ref{r:stablewe}). Similarly, since the left bottom vertical map is a levelwise cofibration and a stable equivalence the same holds for all bottom vertical maps. It then follows that $h_{n, n}\otimes Z\lrar Q_{n+1}(Z)$ is a levelwise cofibration and a stable equivalence over $Y$. We note that by construction the shifted diagram $Q_{n+1}(Z)[n+1]$ is constant on $\Sig_Y(Z)$ (see Lemma~\ref{l:stability} for the definition of the shift functors).

Let us now assume that we have constructed $P_n \lrar Y$ such that $P_n$ is a suspension spectrum up to $n$ and such that the shifted object $P_n[n]$ 
is a constant diagram. 
This is clearly satisfied by $P_0\x{\df}{=} X$. We now define $P_{n+1}$ inductively as the pushout
$$\xymatrix{
h_{n, n}\otimes (P_n)_{n,n}\ar[r]\ar[d] & Q_{n+1}((P_n)_{n,n})\ar[d]\\
P_n\ar[r] & P_{n+1}\poc 
}$$ 
Since the left vertical map becomes an isomorphism after applying the shift $[n]$, so does the right vertical map in the above square. It follows that $P_{n+1}[n+1]$ is constant and that the $n$'th diagonal square of $P_{n+1}$ is homotopy coCartesian by construction. This means that $P_{n+1}$ is a suspension spectrum up to $n$. Furthermore, by construction the map $P_n \lrar P_{n+1}$ is a levelwise cofibration and a stable equivalence which is an isomorphism at $(m,k)$ whenever at least one of $m,k$ is smaller than $n$ or $k=m=n$.
\end{proof}

Taking $Y$ in Lemma~\ref{l:suspensionspectrum} to be the terminal object of $\M^{\NN \times \NN}$ we obtain the following corollary:
\begin{cor}\label{c:suspensionspectrum}
Let $X\in \M$ be a cofibrant object. Then there exists a stable equivalence $\Sigma^\infty X\lrar \bar{\Sigma^\infty X}$ whose codomain is a suspension spectrum and such that the map $X \lrar \bar{\Sigma^\infty X}_{0, 0}$ is a weak equivalence. 
\end{cor}

\begin{rem}\label{r:unique}
Given an injective cofibrant constant spectrum object $X$, Corollary~ \ref{c:suspensionspectrum} provides a stable equivalence $X \lrar X'$ from $X$ to a suspension spectrum which induces an equivalence in degree $(0,0)$. These ``suspension spectrum replacements'' can be organized into a category, and Lemma~\ref{l:suspensionspectrum} can be used to show that the nerve of this category is weakly contractible. We may hence consider a suspension spectrum replacement in the above sense as essentially unique.
\end{rem}

\begin{rem}\label{r:examination}
Examining the proof of Lemma~\ref{l:suspensionspectrum} we see that the suspension spectrum replacement of Corollary~\ref{c:suspensionspectrum} can be chosen to depend functorially on $X$ and the map $X \lrar \bar{\Sigma^\infty X}_{0, 0}$ can be chosen to be an isomorphism.
\end{rem}

\begin{rem}\label{r:prespectrumrep}
A similar but simpler construction replaces any levelwise cofibrant $(\NN\times\NN)$-diagram $X$ by a weakly equivalent pre-spectrum: let $X^{(0)}=X$ and inductively define $X^{(k+1)}$ such that $X^{(k)} \lrar X^{(k+1)}$ is a pushout along 
$$\coprod_{n+m=k, n\neq m} h_{n, m}\otimes X^{(k)}_{n, m} \lrar  \coprod_{n+m=k, n\neq m} h_{n, m}\otimes C(X^{(k)}_{n, m}).$$
The map $X^{(k)}\lrar X^{(k+1)}$ is then an isomorphism below the line $m+n=k$ and replaces the off-diagonal entries on that line by their cones. It is a levelwise cofibration and a stable equivalence, being the pushout of such a map with cofibrant target (see Remark~\ref{r:stablewe}). The (homotopy) colimit of the resulting sequence of stable equivalences yields the desired pre-spectrum replacement.
\end{rem}

\subsection{Differentiable model categories and $\Om$-spectra}
Our goal in this subsection is to give a description of the fibrant replacement of a pre-spectrum, which resembles the classical fibrant replacement of spectra (see~\cite{Hov}, or \cite[Corollary 8.17]{Lur06} for the $\infty$-categorical analogue). This description requires some additional assumptions on the model category at hand, which we first spell out.

Let $f: \I\lrar \M$ be a diagram in a combinatorial model category $\M$. Recall that a cocone $\ovl{f}: \I^{\triangleright} \lrar \M$ over $f$ is called a \textbf{homotopy colimit diagram} if for some projectively cofibrant replacement $f^{\cof} \lrar f$, the composed map
$ \displaystyle\mathop{\colim} f^{\cof}(i) \lrar \displaystyle\mathop{\colim} f(i) \lrar \ovl{f}(\ast) $
is a weak equivalence (where $\ast \in \I^{\triangleright}$ denotes the cone point). A functor $\G: \M \lrar \N$ preserving weak equivalences is said to \textbf{preserve $\I$-indexed homotopy colimits} if it maps $\I^{\triangleright}$-indexed homotopy colimit diagrams to homotopy colimit diagrams.

\begin{define}[cf.~{\cite[Definition 6.1.1.6]{Lur14}}]
Let $\M$ be a model category and let $\NN$ be the poset of non-negative integers as above. We will say that $\M$ is \textbf{differentiable} if for every homotopy finite category $\I$ (i.e., a category whose nerve is a finite simplicial set), the right derived limit functor $\RR\lim: \M^\I \lrar \M$ preserves $\NN$-indexed homotopy colimits. We will say that a Quillen adjunction $\L: \M \adj \N: \R$ is \textbf{differentiable} if $\M$ and $\N$ is differentiable and $\mathbb{R}\R$ preserves $\NN$-indexed homotopy colimits.
\end{define}

\begin{rem}\label{r:diff}
The condition that $\M$ be differentiable can be equivalently phrased by saying that the derived colimit functor $\LL\colim:\M^{\NN} \lrar \M$ preserves finite homotopy limits. This means, in particular, that if $\M$ is differentiable then the collection of $\Om$-spectra in $\M^{\NN \times \NN}$ is closed under $\NN$-indexed homotopy colimits. 
\end{rem}

\begin{exam}
Recall that a combinatorial model category $\M$ is called \textbf{finitely combinatorial} if the underlying category of $\M$ is compactly generated and there exist sets of generating cofibrations and trivial cofibrations whose domains and codomains are compact (see~\cite{RR15}). The classes of fibrations and trivial fibrations, and hence the class of weak equivalences, are then closed under filtered colimits. Such a model category $\M$ is differentiable because filtered colimit diagrams in $\M$ are already filtered homotopy colimit diagrams, while the functor $\colim: \M^{\NN}\lrar \M$ preserves finite limits and fibrations (and hence finite homotopy limits).
\end{exam}

\begin{lem}\label{l:Y}
Let $\M$ be a weakly pointed combinatorial model category and let $f: X \lrar Y$ be a map of pre-spectra such that $X$ is levelwise cofibrant and $Y$ is an injective fibrant $\Om$-spectrum \textbf{at $m$}, i.e. the square
\begin{equation}\label{e:omega1}\vcenter{\xymatrix@R=4.5ex@C=4.5ex{
Y_{m, m}\ar[r]\ar[d] & Y_{m, m+1}\ar[d]\\
Y_{m+1, m}\ar[r] & Y_{m+1, m+1}
}}\end{equation}
is homotopy Cartesian. Then we may factor $f$ as $X \x{f'}{\lrar} X' \x{f''}{\lrar} Y$ such that
\begin{enumerate}[(1)]
\item $f'$ is a levelwise cofibration and a stable equivalence and the map $f'_{n,k}: X_{n,k}\lrar X'_{n,k}$ is a weak equivalence for every $n,k$ except $(n,k) = (m,m)$.
\item $X'$ is an $\Om$-spectrum at $m$.
\end{enumerate}
\end{lem}
\begin{proof}
We first note that we may always factor $f$ as an injective trivial cofibration $X \lrar X''$ followed by an injective fibration $X'' \lrar Y$. Replacing $X$ with $X''$ we may assume without loss of generality that $f$ is an injective fibration.
Let 
$$ X_{m, m}\lrar P\lrar Y_{m,m} \times_{\left[Y_{m,m+1} \times_{Y_{m+1,m+1}}Y_{m+1,m}\right]} \left[X_{m, m+1}\times_{X_{m+1, m+1}} X_{m+1, m}\right] $$
be a factorization in $\M$ into a cofibration followed by a trivial fibration. By our assumption on $Y$ the map $Y_{m,m} \lrar Y_{m,m+1} \times_{Y_{m+1,m+1}}Y_{m+1,m}$ is a trivial fibration and hence the composed map $P \lrar X_{m, m+1}\times_{X_{m+1, m+1}} X_{m+1, m}$ is a trivial fibration as well.
Associated to the cofibration $j: X_{m, m}\lrar P$ is now a square of $(\NN\times\NN)$-diagrams
\begin{equation}\label{e:PQ}
\vcenter{\xymatrix{
\left(h_{m, m+1}\coprod_{h_{m+1, m+1}} h_{m+1, m}\right)\otimes X_{m, m}\ar[r]\ar[d] & \left(h_{m, m+1}\coprod_{h_{m+1, m+1}} h_{m+1, m}\right)\otimes P\ar[d]\\
h_{m, m}\otimes X_{m, m}\ar[r] & h_{m, m}\otimes P
}}
\end{equation}
The rows of these diagarms are stable equivalences and levelwise cofibrations between levelwise cofibrant objects. It follows that the induced map $i_m\Box j: Q\lrar h_{m, m}\otimes P$ from the (homotopy) pushout to $h_{m, m}\otimes P$ is a stable equivalence and a levelwise cofibration (see Remark~\ref{r:stablewe}). One can easily check that $i_m\square j$ is an isomorphism in every degree, except in degree $(m, m)$ where it is the inclusion $X_{m, m}\lrar P$. We now define $X'$ as the pushout
$$\xymatrix{
Q\ar[r]\ar[d] & h_{m, m}\otimes P\ar[d]\\
X\ar[r] & X'
}$$
where the left vertical map is the natural map. Since $Q$ and $X$ are levelwise cofibrant, the resulting map $X\lrar X'$ is a stable equivalence and an isomorphism in all degrees, except in degree $(m, m)$ where it is the cofibration $X_{m, m}\lrar P$. we now see that the map $X \lrar X'$ satisfies properties (1) and (2) above by construction. 
\end{proof}

\begin{cor}\label{c:Ln}
Let $\M$ be a weakly pointed combinatorial model category and let $f: X \lrar Y$ be a map in $\M^{\NN\times\NN}$ between pre-spectra such that $X$ is levelwise cofibrant and $Y$ is an injective fibrant $\Om$-spectrum \textbf{below $n$}, i.e., it is an $\Om$-spectrum at $m$ for every $m<n$. 
Then we may factor $f$ as $X \x{f'}{\lrar} L_nX \x{f''}{\lrar} Y$ such that $f'$ is a levelwise cofibration and a stable equivalence, $L_nX$ is an $\Om$-spectrum below $n$ and the induced map $f'[n]: X[n]\lrar L_nX[n]$ 
is a levelwise weak equivalence of pre-spectra. In particular, if the induced map $f[n]:X[n] \lrar Y[n]$ is already a levelwise weak equivalence then $f$ is a stable equivalence.
\end{cor}
\begin{proof}
Apply Lemma~\ref{l:Y} consecutively for $m=n-1,...,0$ to construct the factorization $X \lrar L_nX \lrar Y$ with the desired properties. Note that if $f[n]: X[n] \lrar Y[n]$ is a levelwise equivalence then the induced map $L_nX[n] \lrar Y[n]$ is a levelwise equivalence and since both $L_nX$ and $Y$ are $\Om$-spectra below $n$ the map $L_nX \lrar Y$ must be a levelwise equivalence. It then follows that $f: X \lrar Y$ is a stable equivalence.
\end{proof}

\begin{cor}\label{c:fibrantreplacement}
Let $\M$ be a weakly pointed differentiable combinatorial model category and let $f: X \lrar Y$ be a map in $\M^{\NN\times\NN}$ such that $X$ is levelwise cofibrant pre-spectrum and $Y$ is an injective fibrant $\Om$-spectrum. Then there exists a sequence of levelwise cofibrations and stable equivalences
$$
X \lrar L_1X \lrar L_2X \lrar  \cdots 
$$
over $Y$ such that for each $n$ the map $X[n]\lrar L_nX[n]$ is a levelwise weak equivalence and $L_nX$ is an $\Omega$-spectrum below $n$. Furthermore, the induced map $X \lrar L_{\infty}X \x{\df}{=} \colim L_nX$ is a stable equivalence and $L_{\infty}X$ is an $\Om$-spectrum.
\end{cor}
\begin{proof}
Define the objects $L_nX$ inductively by requiring $L_nX\lrar L_{n+1}X$ to be the map from $L_nX$ to an $\Omega$-spectrum below $n+1$ constructed in Corollary~\ref{c:Ln}. The resulting sequence is easily seen to have all the mentioned properties.

Since all the maps $L_nX \lrar L_{n+1}X$ are levelwise cofibrations between levelwise cofibrant objects it follows that the map $X\lrar L_\infty X$ is the homotopy colimit in $\M^{\NN \times \NN}$ of the maps $X \lrar L_nX$. Since the collection of stable equivalences between pre-spectra is closed under homotopy colimits we may conclude that the map $X \lrar L_\infty X$ is a stable equivalence between pre-spectra. The assumption that $\M$ is differentiable implies that for each $m$ the collection of $\Om$-spectra at $m$ is closed under sequential homotopy colimits. We may therefore conclude that $L_\infty X$ is an $\Omega$-spectrum at $m$ for every $m$, i.e., an $\Om$-spectrum. 
\end{proof}

\begin{rem}\label{r:om-oo}
Since the map $X_{n,n} \lrar (L_nX)_{n,n}$ is a weak equivalence in $\M$ and $L_nX$ is a pre-spectrum and an $\Om$-spectrum below $n$ it follows that the space $(L_nX)_{0, 0}$ is a model for $n$-fold loop object $\Omega^n X_{n, n}$ in $\M$. The above result then asserts that for any pre-spectrum $X$, its $\Om$-spectrum replacement $L_{\infty}X$ is given in degree $(k,k)$ by $\hocolim_n \Omega^n X_{k+n, k+n}$. In particular $\mathbb{R}\Om^{\infty}X \simeq \hocolim_n\Om^nX_{n,n}$.
\end{rem}

\begin{cor}\label{c:preservationofweakequivalences}
Let $\R: \M\lrar \N$ be a differentiable right Quillen functor between weakly pointed combinatorial model categories. Then the right derived Quillen functor $\RR\R^{\NN \times \NN}: \M^{\NN\times\NN}_{\Reedy} \lrar \N^{\NN\times\NN}_{\Reedy}$ preserves stable equivalences between pre-spectra. If in addition $\RR\R$ detects weak equivalences then $\RR\R^{\NN\times\NN}$ detects stable equivalences between pre-spectra.
\end{cor}
\begin{proof}
Let $f: X \lrar Y$ be a stable equivalence between pre-spectra. We may assume without loss of generality that $X$ is levelwise cofibrant.
Let
$$
Y \lrar L_1Y \lrar L_2Y \lrar  \cdots 
$$
be constructed as in Corollary~\ref{c:fibrantreplacement} with respect to the map $Y \lrar \ast$ and let $Y_{\infty} = \colim_nL_nY$. Similarly, let
$$
X \lrar L_1X \lrar L_2X \lrar  \cdots 
$$
be a sequence as in Corollary~\ref{c:fibrantreplacement} constructed with respect to the map $X \lrar Y_\infty$, and let $X_\infty= \colim_nL_nX$. Since $L_nX$ is an $\Om$-spectrum below $n$ it follows that $\RR\R^{\NN \times \NN}(L_nX)$ is an $\Om$-spectrum below $n$. Furthermore, since the map $\RR\R^{\NN \times \NN}(X)[n] \lrar \RR\R^{\NN \times \NN}(L_nX)[n]$ is a levelwise equivalence it follows from the final part of Corollary~\ref{c:Ln} that the map $\RR\R^{\NN \times \NN}(X) \lrar \RR\R^{\NN \times \NN}(L_nX)$ is a stable equivalence. By the same argument the map $\RR\R^{\NN \times \NN}(Y) \lrar \RR\R^{\NN \times \NN}(L_nY)$ is a stable equivalences. Since the maps $L_nX \lrar L_{n+1}X$ are levelwise cofibrations between levelwise cofibrant objects it follows that $X_\infty \simeq \hocolim L_nX$ and $Y_\infty \simeq \hocolim_nL_nY $. Since $\RR \R$ preserves sequential homotopy colimits by assumption we may conclude that the maps
$\RR\R^{\NN \times \NN}(X) \lrar \RR\R^{\NN \times \NN}(X_\infty)$ and $\RR\R^{\NN \times \NN}(Y) \lrar \RR\R^{\NN \times \NN}(Y_\infty)$ are stable equivalences. Now since $X_\infty \lrar Y_\infty$ is a stable equivalence between $\Om$-spectra it is also a levelwise weak equivalence. We thus conclude that the map $\RR\R^{\NN \times \NN}(X_\infty) \lrar \RR\R^{\NN \times \NN}(Y_\infty)$ is a levelwise equivalence. The map $\RR\R^{\NN \times \NN}(X) \lrar \RR\R^{\NN \times \NN}(Y)$ is hence a stable equivalence in $\N^{\NN \times \NN}$ by the $2$-out-of-$3$ property.
\end{proof}

\begin{cor}\label{c:new}
Let $\L:\M \adj \N:\R$ be a differentiable Quillen pair of weakly pointed left proper combinatorial model categories and let $n\geq 0$ be a natural number.
\begin{enumerate}[(1)]
\item 
If the derived unit $u_X: X\lrar  \RR\R (\L X)$ either has the property that $\Om^n u_X$ is an equivalence for every cofibrant $X$ or $\Sig^nu_X$ is an equivalence for every cofibrant $X$, then the derived unit of $\Sp(\L) \dashv \Sp(\R)$ is weak equivalence for every levelwise cofibrant pre-spectrum. 
\item 
If the derived counit $\nu_X: \LL\L(\R Y)\lrar Y$ either has the property that $\Om^n \nu_X$ is an equivalence for every fibrant $Y$ or $\Sig^n\nu_Y$ is an equivalence for every fibrant $Y$, then the derived counit of $\Sp(\L) \dashv \Sp(\R)$ is weak equivalence for every levelwise fibrant pre-spectrum.
\end{enumerate}
\end{cor}
\begin{proof}
We will only prove the first claim; the second claim follows from a similar argument. 
Let $A\in\M^{\NN\times\NN}$ be a levelwise cofibrant pre-spectrum object in $\M$. Since $\R$ is differentiable 
we have by Corollary~\ref{c:preservationofweakequivalences} that $\RR\R^{\NN\times\NN}$ preserves stable equivalences between pre-spectra. It follows that the derived unit map is 
is given levelwise by the derived unit map of the adjunction $\L\dashv\R$. In particular, if each component of this map becomes an equivalence upon applying $\Sig^n$, then the entire unit map becomes a levelwise equivalence after suspending $n$ times (recall that suspension in $\Sp(\M)$, like all homotopy colimits, can be computed levelwise). Since $\Sp(\M)$ is stable this means that the derived unit itself is an equivalence. 

Now assume that $u_X$ becomes an equivalence after applying $\Om^n$. Since $\L\dashv \R$ is a Quillen adjunction between weakly pointed model categories, the above map is a map of pre-spectra. It therefore suffices to check that the induced map
$$ A^{\fib} \lrar \R\left(\L(A)^{\Reedy-\fib}\right)^{\fib} $$
on the (explicit) fibrant replacements provided by Corollory~\ref{c:fibrantreplacement} is a levelwise equivalence. By Remark~\ref{r:om-oo} this map is given at level $(k,k)$ by the induced map
$$ \hocolim_i \Om^i A_{k+i, k+i} \lrar \hocolim_i \Om^i \RR\R(\L(A_{k+i,k+i})).$$
We now observe that $\Om^iA_{k+i,k+i} \lrar \Om^i \RR\R(\L(A_{k+i,k+i}))$ is a weak equivalence for all $i \geq n$ by our assumption, and so the desired result follows.
\end{proof}

\section{The tangent bundle}
\label{ss:cotangent}\label{s:global}

\subsection{The tangent bundle as a relative model category}\label{s:modcatfib}
The tangent bundle $\T\M$ can informally be thought of as describing the homotopy theory of parameterized spectra in $\M$, with varying base objects. Accordingly, one can consider $\T\M$ itself as being parameterized by the objects of $\M$: for every object $A\in \M$, there is a full subcategory of the tangent bundle consisting of spectra parameterized by $A$. More precisely, the tangent bundle fits into a commuting triangle of right Quillen functors
\begin{equation}\label{d:tangentfib}\vcenter{\xymatrix{
\T\M\ar[rd]_\pi\ar[rr]^{\Om^\infty_+} & & \M^{[1]}\ar[ld]^{\codom}\\
& \M & 
}}\end{equation}
where the functor $\Om^\infty_+$ sends a diagram $X: (\NN\times \NN)_\ast\lrar \M$ to its restriction $X(0, 0)\lrar X(\ast)$ and the functor ``$\codom$'' takes the codomain of an arrow in $\M$.

The functors $\pi: \T\M\lrar \M$ and $\codom: \M^{[1]}\lrar \M$ have various favorable properties. For example, in addition to being right adjoint functors, they are left adjoints as well, with right adjoints given by the formation of constant diagrams. More importantly, they are both Cartesian and coCartesian fibrations, with fiber over $A\in \M$ given by the categories $\big(\M_{A//A}\big)^{\NN\times \NN}$ and $\M_{/A}$, respectively.

The purpose of this section is to show that this behaviour persists at the homotopical level. In \S\ref{s:modcatfib} we discuss how the functor $\pi$ behaves like a fibration of model categories (cf.\ \cite{HP}) which is classified by a suitable diagram of model categories and left Quillen functors between them. In \S\ref{s:oo-stab}, we show that the triangle of right Quillen functors \eqref{d:tangentfib} realizes $\T\M$ as a model for the tangent $\infty$-category of the $\infty$-category underlying $\M$.

Recall that a suitable version of the classical Grothendieck correspondence asserts that the data of a (pseudo-)functor from an ordinary category $\C$ to the $(2,1)$-category of categories and adjunctions is equivalent to the data of a functor $\D \lrar \C$ which is simultaneously a Cartesian and a coCartesian fibration. This result admits a model categorical analogue, developed in~\cite{HP}, classifying certain fibrations $\N\lrar \M$ of categories equipped with three wide subcategories 
$$
\W_\M,\Cof_\M,\Fib_\M \subseteq \M
$$
(similarly for $\N$). We will refer to such a category equipped with three wide subcategories as a \textbf{pre-model category}. The morphisms in $\W_\M$, $\Cof_\M$, $\Fib_\M$, $\Cof_\M \cap \W_\M$ and $\Fib_\M \cap \W_\M$ will be called weak equivalences, cofibrations, fibrations, trivial cofibrations and trivial fibrations respectively. 

\begin{defn}\label{d:relative-model}
Let $\M,\N$ be two pre-model categories and $\pi:\N \lrar \M$ a (co)Cartesian fibration which preserves the classes of (trivial) cofibrations and (trivial) fibrations. We will say that $\pi$ exhibits $\N$ as a \textbf{model category relative to $\M$} if the following conditions are satisfied:
\begin{enumerate}
\item
$\pi: \N \lrar \M$ is (co)complete, i.e., admits all relative limits and colimits. 
\item
Let $f: X \lrar Y$ and $g: Y \lrar Z$ be morphisms in $\N$. If two of $f,g,g\circ f$ are in $\W_\N$ and if the image of the third is in $\W_\M$ then the third is in $\W_\N$.
\item
$(\Cof_\N\cap \W_\N,\Fib_\N)$ and $(\Cof_\N,\Fib_\N\cap \W_\N)$ are $\pi$-weak factorization systems relative to $(\Cof_\M \cap \W_\M,\Fib_\M)$ and $(\Cof_\M,\Fib_\M \cap \W_\M)$ respectively. In other words, every lifting/factorization problem in $\N$ which has a solution in $\M$ admits a compatible solution in $\N$ (see \cite[Definition 5.0.2]{HP} for the full details).
\end{enumerate}
In this case we will also say that $\pi$ is a \textbf{relative model category}.
\end{defn}

\begin{rem}
In~\cite{HP} the authors consider the notion of a relative model category in the more general case where $\pi$ is not assumed to be a (co)Cartesian fibration. However, for our purposes we will only need to consider the more restrictive case, which is also formally better behaved (for example, it is closed under composition, see~\cite{HP-arrtum}).
\end{rem}


\begin{rem}\label{r:retractarg}
If $\N \lrar \M$ is a relative model category then the cofibrations and trivial fibrations in $\N$ determine each other, in the following sense: if $f: X\lrar Y$ is a map such that $\pi(f)$ is a cofibration in $\M$, then $f$ is a cofibration in $\N$ if and only if it has the relative left lifting property against all trivial fibrations in $\N$ which cover identities. Indeed, this follows from the usual retract argument, where one factors $f$ as a cofibration $i: X\lrar \tilde{Y}$ over $\pi(f)$, followed by a trivial fibration $p: \tilde{Y}\lrar Y$ over the identity and shows that $f$ is a retract of $i$ (over $\pi(f)$) using that $f$ has the relative left lifting property against $p$.
\end{rem}

\begin{exam}
If $\pi: \N \lrar \M$ is a relative model category and $\M$ is a model category then $\N$ is a model category and $\pi$ is both a left and right Quillen functor.
\end{exam}

If $\pi: \N\lrar \M$ is a relative model category, then the functor $\emptyset: \M\lrar \N$ preserves all (trivial) cofibrations and the functor $\ast: \M\lrar \N$ preserves all (trivial) fibrations. Orthogonally, for every object $A\in \M$, the (co)fibrations and weak equivalences of $\N$ that are contained in the fiber $\N_A$, together determine a model structure on $\N_A$: indeed, the relative factorization, lifting and retract axioms in particular imply these axioms \emph{fiberwise}. 

Since $\pi: \N \lrar \M$ is a (co)Cartesian fibration every map $f: A\lrar B$ in $\M$ determines an adjoint pair
$$
f_!: \N_A\adj \N_B : f^*.
$$
This adjunction is a Quillen pair, as one easily deduces from the following result:
\begin{lem}\label{l:cofibrationsinrelmod}
Let $\pi: \N\lrar \M$ be a relative model category and let $f: X\lrar Y$ be a map in $\N$. Then $f$ is a (trivial) cofibration if and only if $\pi(f)$ is a (trivial) cofibration in $\M$ and the induced map $\pi(f)_!X\lrar Y$ is a (trivial) cofibration in $\N_{\pi(Y)}$.
\end{lem}
\begin{proof}
Consider a lifting problem in $\N$ of the form
$$\xymatrix{
X\ar[r]\ar[d]_f & Z\ar[d]^{}="s" & & A\ar[r]\ar[d]_{\pi(f)}="t" & C\ar[d]\\ 
Y\ar[r]\ar@{..>}[ru]_{\smash{\tilde{g}}} & W & & B\ar[ru]_g\ar[r] & D \ar@{|->}"s"+<15pt, 0pt>;"t"-<15pt, 0pt>^{\pi}
}$$
together with a diagonal lift of its image in $\M$, as indicated. Finding the desired diagonal lift $\tilde{g}$ covering $g$ is equivalent to finding a diagonal lift $g'$ covering $g$ for the diagram
$$\vcenter{\xymatrix{
\pi(f)_!X\ar[d]\ar[r] & Z\ar[d]\\
Y\ar[r]\ar@{..>}[ru]_{g'} & W.
}}$$
It follows that a map $f: X\lrar Y$ has the relative left lifting property against all trivial fibrations in $\N$ if and only if the induced map $\pi(f)_!X\lrar Y$ does. In other words (see Remark \ref{r:retractarg}), if $\pi(f)$ is a cofibration, then $f$ is itself a cofibration in $\N$ iff $\pi(f)_!X\lrar Y$ is a cofibration in $\N$ and the result follows. A similar argument applies to the trivial cofibrations.
\end{proof}
\begin{rem}\label{r:cocartesianliftofcof}
In particular, Lemma \ref{l:cofibrationsinrelmod} implies that any coCartesian lift of a (trivial) cofibration in $\M$ is a (trivial) cofibration in $\N$ (see also \cite[Lemma 5.0.11]{HP} for an alternative proof). Dually, any Cartesian lift of a (trivial) fibration in $\M$ is a (trivial) fibration in $\N$.
\end{rem}


%
In general, the Quillen pair associated to a weak equivalence in $\M$ need not be a Quillen equivalence; to guarantee this, it suffices to require the relative model category $\pi: \N\lrar \M$ to satisfy the following additional conditions:
\begin{defn}[{\cite[Definition 5.0.8]{HP}}]\label{d:modfib}
Let $\pi: \N \lrar \M$ be a (co)Cartesian fibration which exhibits $\N$ as a relative model category over $\M$. We will say that $\pi$ is a \textbf{model fibration} if it furthermore satisfies the following two conditions:
\begin{itemize}
\item[(a)]
If $f: X \lrar Y$ is a $\pi$-coCartesian morphism in $\N$ such that $X$ is cofibrant in $\N_{\pi(X)}$ and $\pi(f) \in \W_\M$ then $f \in \W_\N$.
\item[(b)]
If $f: X \lrar Y$ is a $\pi$-Cartesian morphism in $\N$ such that $Y$ is fibrant in $\N_{\pi(Y)}$ and $\pi(f)$ is in $\W_\M$ then $f \in \W_\N$.
\end{itemize}
\end{defn}
\begin{rem}
These two conditions are equivalent to the following assertion: let $f: X\lrar Y$ be a map in $\M$ covering a weak equivalence in $\M$ such that $X\in \N_{\pi(X)}$ is cofibrant and $Y\in \N_{\pi(Y)}$ is fibrant. Then $f$ is a weak equivalence iff the induced map $\pi(f)_!X\lrar Y$ is an equivalence in $\N_{\pi(Y)}$ iff $X\lrar \pi(f)^*Y$ is an equivalence in $\N_{\pi(X)}$. In particular, $f_!\dashv f^*$ is a Quillen equivalence for any $f\in \W_\M$.
\end{rem}
The main result of \cite{HP} asserts that such model fibrations are completely classified by the functor $\M\lrar \ModCat$ sending $A\mapsto \N_A$ (and $f$ to the left Quillen functor $f_!$) and that conversely, any functor $\M\lrar\ModCat$  determines a model fibration as soon as it is \textbf{relative} (i.e.\ weak equivalences are sent to Quillen equivalences) and \textbf{proper} (see loc.\ cit.\ for more details).

Our goal in this section is to prove the following theorem:
\begin{thm}\label{t:fib}
Let $\M$ be a left proper combinatorial model category and let 
$$
\pi: \T\M\lrar \M
$$
be the projection evaluating an $(\NN\times\NN)_\ast$-diagram on the basepoint $\ast$. Then $\pi$ exhibits $\T\M$ as relative model category over $\M$. Furthermore,
the restriction $\T\M \times_{\M} \M^{\fib} \lrar \M^{\fib}$ to the full subcategory $\M^{\fib}\subseteq \M$ of fibrant objects is a model fibration, classified by
$$
\F: \M^{\fib} \lrar \ModCat; \hspace{4pt} A\mapsto \Sp(\M_{A//A}).
$$
\end{thm}
Let us start by showing that $\pi: \T\M\lrar \M$ is a relative model category. Since $\T\M$ is a left Bousfield localization of the Reedy model structure, this will following from the following two results:
\begin{lem}\label{l:relative}
Let $\M$ be a model category, $\J$ a Reedy category and $n \geq 0$ a given integer. If $\J_{\leq n} \subseteq \J$ denotes the full subcategory spanned by the objects of degree $\leq n$, then the restriction functor 
\begin{equation}\label{e:i}
\M^{\J}_{\Reedy} \lrar \M^{\J_{\leq n}}_{\Reedy}
\end{equation} 
is a relative model category.
\end{lem}
\begin{proof}
Since the domain and codomain of~\eqref{e:i} are model categories the relative $2$-out-of-$3$ property and relative closures under retracts are automatic. Furthermore, it is straightforward to check that since $\M$ is (co)complete and $\I_{\leq n} \hrar \I$ is a fully-faithful inclusion then the restriction functor $\M^{\J} \lrar \M^{\J_{\leq n}}$ is a (co)Cartesian fibration which is relatively (co)complete.

To verify that~\eqref{e:i} has relative factorizations and relative lifting properties, one proceeds by induction, analogous to the proof of the existence of the Reedy model structure: given a factorization (lifting) problem with a solution in degrees $\leq n$, the problem of extending this solution to degrees $\leq n+1$ is equivalent to a certain set of factorization (lifting) problems in $\M$, involving $(n+1)$-st latching and matching objects. Inductively choosing such factorizations (lifts) in $\M$ produces the desired compatible factorization (lift) in $\M^{\J}$.
\end{proof}
\begin{pro}\label{p:localization}
Let $\pi: \N\lrar \M$ be a (co)Cartesian fibration which exhibits $\N$ as a  relative model category over a model category $\M$, and suppose that $\N$ is a left proper combinatorial model category.
If $S$ is a set of maps in $\N$, then the functor $\pi: L_S\N\lrar\M$ is a relative model category as soon as it preserves the $S$-local trivial cofibrations.
\end{pro}
\begin{proof}
The relative 2-out-of-3 and retract axioms are obviously satisfied, since $L_S\N$ and $\M$ are model categories. Since the cofibrations and trivial fibrations of $L_S\N$ agree with those of $\N$, they still satisfy the relative factorization and lifting axioms. It remains to verify the relative factorization and lifting axioms for the classes of $S$-local trivial cofibrations and $S$-local fibrations.

For the lifting axiom, consider a diagram
$$\xymatrix@C=2.3pc@R=2.3pc{
X\ar[r]\ar@{^{(}->}[d]_-{\smash{\tilde{i}}}^-{\sim_S} & Z\ar@{->>}[d]^{\tilde{p}}="s" & & A\ar[r]\ar[d]_{i}="t" & C\ar[d]^p\\ 
Y\ar[r]\ar@{..>}[ru]_{\smash{\tilde{f}}} & W & & B\ar[ru]_-f\ar[r] & D \ar@{|->}"s"+<15pt, 0pt>;"t"-<15pt, 0pt>^{\pi}
}$$
together with a diagonal lift of its image in $\M$, as indicated. Here $\tilde{i}$ is an $S$-local trivial cofibration and $\tilde{p}$ is an $S$-local fibration, so that their images in $\M$ are a trivial cofibration (by assumption) and a fibration, respectively.

Arguing as in the proof of Lemma~\ref{l:cofibrationsinrelmod} we see that to find the desired diagonal lift $\tilde{f}$ covering $f$, it suffices to find a diagonal lift for the diagram
\begin{equation}\label{e:fiberlift}\vcenter{\xymatrix{
f_!i_!X\ar[d]\ar[r] & Z\ar[d]\\
f_! Y\ar[r]\ar@{..>}[ru] & p^*W
}}\end{equation}
in the fiber $\N_C$ over $C$. Since the entire diagram is already contained in $\N_C$, any diagonal lift in $\N$ will automatically be contained in the fiber $\N_C$. It therefore suffices to verify that the map $\alpha: f_!i_!X\lrar f_!Y$ is an $S$-local trivial cofibration, while the map $\beta: Z\lrar p^*W$ is an $S$-local fibration.

To see that $\alpha$ is an $S$-local trivial cofibration, observe first that it arises as the pushout of the map $i_!X\lrar Y$ along the cocartesian edge $i_!X\lrar f_!i_!X$ covering $f$. To see that $i_!X\lrar Y$ is an $S$-local trivial cofibration, note that it fits into a sequence
$$
\tilde{i}: X\lrar i_!X\lrar Y
$$
where $X\lrar i_!X$ is a coCartesian lift of the trivial cofibration $i$ and hence a trivial cofibration in $\N$, by Remark \ref{r:cocartesianliftofcof}. The map $i_!X\lrar Y$ is then a cofibration by Lemma \ref{l:cofibrationsinrelmod} and an $S$-local weak equivalence by the 2-out-of-3 property.

Similarly, the map $\beta: Z\lrar p^*W$ fits into a sequence
$$\xymatrix{
\tilde{p}: Z\ar[r] & p^*W\ar[r] & W
}$$
whose composite is the $S$-local fibration $\tilde{p}$ and where $p^*W\lrar W$ is a cartesian lift of the fibration $p: C\lrar D$ in $\M$. In particular, $\beta$ is a fibration in $\N$, before localizing at $S$ (by the dual of Lemma \ref{l:cofibrationsinrelmod}). On the other hand, $p^*W\lrar W$ fits into a pullback square
$$\xymatrix{
p^*W\ar[r]\ar[d] & W\ar[d]\\
\ast_C\ar[r] & \ast_D.
}$$
Since $\ast: \M\lrar L_S\N$ is right Quillen by assumption and $C\lrar D$ is a fibration, the map $p^*W\lrar W$ is an $S$-local fibration. To conclude that $\beta: Z\lrar p^*W$ is an $S$-local fibration as well, we can consider it as a map
$$
\beta: (Z\lrar W)\lrar (p^*W\lrar W)
$$
in the over-category $\N_{/W}$. Note that the slice model structure on $\N_{/W}$ induced from $L_S\N$ is a left Bousfield localization of the slice model structure induced from $\N$. The map $\beta$ is now a (non-local) fibration between two local objects in $\N_{/W}$, hence it is a local fibration itself \cite[Proposition 3.3.16]{Hir}. In particular, $\beta: Z\lrar p^*W$ is an $S$-local fibration, and we conclude that the desired lift in \eqref{e:fiberlift} exists.

Next, let $f: X\lrar Y$ be a map in $\N$ with a factorization of its image in $\M$ as a trivial cofibration, followed by a fibration
$$\xymatrix{
\Big(X\ar[r]^f & Y\Big)\ar@{|->}[r]^{\pi} & \Big(A\ar@{^{(}->}[r]_i^\sim & \tilde{A}\ar@{->>}[r]^p & B\Big)
}$$
We have to provide a compatible factorization of $f$. To this end, decompose $f$ as
$$\xymatrix{
X\ar[r] & i_!X\ar[r]^{f'} & p^*Y\ar[r] & Y
}$$
where the maps $X\lrar i_!X$ and $p^*Y\lrar Y$ are cocartesian and cartesian lifts of $i$ and $p$, respectively. By Remark \ref{r:cocartesianliftofcof}, the map $X\lrar i_!X$ is a trivial cofibration (even before Bousfield localization), while $p^*Y\lrar Y$ is an $S$-local fibration (being the base change of the $S$-local fibration $\ast_{\tilde{A}}\lrar \ast_B$). It therefore suffices to provide a factorization within the fiber $\N_{\tilde{A}}$ of the map $f'$ into an $S$-local trivial cofibration, followed by an $S$-local fibration.

In other words, we can reduce to the case where $f: X\lrar Y$ is contained in a fiber $\N_A$. Let
$$\xymatrix{
X\ar@{^{(}->}[r]^{\tilde{i}}_{\sim_S} & \tilde{X}\ar@{->>}[r]^{\tilde{p}} & Y
}$$
be a factorization of this map into an $S$-local trivial cofibration, followed by an $S$-local fibration. The image of this factorization is a factorization
$$\xymatrix{
A\ar@{^{(}->}[r]^{i}_{\sim_S} & \tilde{A}\ar@{->>}[r]^{p} & A
}$$
of the identity map into a trivial cofibration $i$, followed by a trivial fibration $p$. Now consider the following diagram:
$$\xymatrix{
X\ar@{^{(}->}[r]^{\sim}\ar@{=}[rd] & i_!X\ar@{^{(}->}[r]^-{\sim_S}\ar[d] & \tilde{X}\ar[d]\ar@{->>}[rd]^{\tilde{p}} \\ 
& X=p_!i_!X\ar@{^{(}->}[r]_-{\sim_S} & p_!\tilde{X}\ar[r] & Y
}$$
Here the top row is the factorization of $\tilde{i}$ as a $\pi$-coCartesian arrow, followed by an arrow in $\N_{\tilde{A}}$. The vertical map $i_!X\lrar p_!i_!X$ is a $\pi$-coCartesian lift of the map $p$, the middle square is a pushout in $\N$ and the map $p_!\tilde{X}\lrar Y$ is the universal map. The bottom row provides a factorization of the map $f: X\lrar Y$ within the fiber $\N_A$.

Since $i$ is a trivial cofibration in $\M$, the cocartesian arrow $X\lrar i_!X$ is a trivial cofibration before Bousfield localization. Since the top horizontal composite is $\tilde{i}$, it follows that the map $i_!X\lrar \tilde{X}$ is an $S$-local trivial cofibration. Its pushout $X\lrar p_!\tilde{X}$ is then an $S$-local trivial cofibration as well. Furthermore, the map $i_!X\lrar X$ is a weak equivalence in $\N$ (before Bousfield localization) by the 2-out-of-3 property. Since $\N$ is left proper, the pushout $\tilde{X}\lrar p_!\tilde{X}$ is a weak equivalence in $\N$ as well.

The desired factorization of $f$ within the fiber $\N_A$ is now given by
$$\xymatrix{
X\ar@{^{(}->}[r]^{\sim_S} & p_!\tilde{X}\ar@{^{(}->}[r]^{\sim} & X'\ar@{->>}[r]^q & Y.
}$$
Here $p_!\tilde{X}\lrar X'\lrar Y$ is a factorization within $\N_A$ into a trivial cofibration in $\N$ (before localization), followed by a fibration in $\N$ (before localization). Such a factorization exists because $\N_A$ is a model category before left Bousfield localization. The map $X\lrar X'$ is an $S$-local trivial cofibration within $\N_A$, so it remains to verify that the map $q: X'\lrar Y$ is not just a fibration in $\N$, but also an $S$-local fibration. But now observe that the map $q$ fits into a commuting triangle
$$\xymatrix{
\tilde{X}\ar@{->>}[rd]_{\tilde{p}}\ar[r]^\sim & p_!X\ar[r]^\sim & X'\ar@{->>}[ld]^q\\
& Y
}$$
where the two horizontal maps are weak equivalences in $\N$ (before localization). Since the map $\tilde{p}: \tilde{X}\lrar Y$ was an $S$-local fibration, we deduce  that the fibration $q$ is (non-locally) weakly equivalent to an $S$-local fibration. This implies that $q$ is itself an $S$-local fibration as well \cite[Proposition 3.3.15]{Hir}, so that $X\lrar X'\lrar Y$ is a fiberwise factorization of $f$ into an $S$-local trivial cofibration, followed by an $S$-local fibration.
\end{proof}
Applying Lemma \ref{l:relative} and Proposition \ref{p:localization} to the situation where $\N= \M^{(\NN\times\NN)_\ast}_{\Reedy}$, $\pi\colon \N\lrar \M$ evaluates at the unique object $\ast$ of degree $0$ and $L_S\N=\T\M$, one finds that $\pi: \T\M\lrar \M$ is a relative model category.

We will now verify that $\pi$ is a model fibration when restricted to the fibrant objects of $\M$. 
%
%
%
Let us start by verifying this before left Bousfield localization.
\begin{pro}\label{p:reedygivesmodfib}
Let $\J$ be a Reedy category and let $\J_\ast$ be the induced Reedy category obtained by adding a zero object, which is the unique object of degree $0$ in $\J_\ast$. If $\M$ is a left proper model category, then the base changed relative model category
\begin{equation}\label{e:ii}
\M^{\J}_{\Reedy} \times_{\M} \M^{\fib} \lrar \M^{\fib}
\end{equation}
is a model fibration and the functor $\M^{\fib} \lrar \ModCat$ which classifies it (under the equivalence of~\cite[Theorem 5.0.10]{HP}) is given by $A \mapsto (\M_{A//A})^{\J}_{\Reedy}$.
\end{pro}
\begin{proof}
Since~\eqref{e:ii} is the restriction of the relative model category \eqref{e:i}, it is a relative model category as well. Furthermore it is clear that $\pi$ is (co)Cartesian fibration which is classified by the functor $\M^{\fib} \lrar \AdjCat$ given by $A \mapsto \left(\M_{A//A}\right)^{I}$, where for every $f: A \lrar B$ the induced adjunction $\left(\M_{A//A}\right)^{\J} \lrar \left(\M_{B//B}\right)^{\J}$ is defined by 
$$
f_!(A \lrar X_\bullet \lrar A) = B \lrar X_\bullet \coprod_{A} B \lrar B
$$ and 
$$
f^*(B \lrar Y_\bullet \lrar B) = A \lrar Y_\bullet\times_{B} A \lrar A .
$$

It remains to verify conditions (a) and (b) of Definition \ref{d:modfib}. To prove (a), let 
$\iota_!: \M \lrar \M^{\J_\ast}$ be the left Kan extension functor along the inclusion $\iota:\{\ast\} \subseteq \J_\ast$. Consider a functor $\F:\J_\ast \lrar \M$ such that $\iota_!\F(\ast) \lrar \F$ is a Reedy cofibration (this is the condition that $\F$ is cofibrant in its fiber over $\M$). Let $\vphi: \F(\ast) \lrar B$ be a weak equivalence in $\M$ and let 
$$ \psi: \F \lrar \F \coprod_{\iota_!\F(\ast)} \iota_! B $$ 
be its coCartesian lift. We need to prove that $\psi$ is a weak equivalence. 
We now observe that since $\ast$ is initial in $\J_\ast$ the functor $\iota_!$ sends $A \in \M$ to the constant functor with value $A$. We hence just need to show that the map $\psi(x): \F(x) \lrar \F(x) \coprod_{\F(\ast)}B$ is a weak equivalence for every $x \in \J$. But this now follows from the fact that the map $\F(\ast) \lrar B$ is a weak equivalence, the map $\F(\ast) \lrar \F(x)$ is a cofibration, and $\M$ is left proper. The proof of (b) is similar, using that $\F(\ast)$ is assumed to be fibrant.

We now prove that this model fibration is classified by the functor
$$
\M^{\fib} \lrar \ModCat; \hspace{4pt} A \mapsto (\M_{A//A})^{\J}_{\Reedy}.
$$
In particular, we need to show that the induced model structure on $\F(A) = \left(\M_{A//A}\right)^{\J}$ coincides with the Reedy model structure.

Let $\vphi: \F \lrar \G$ be a map in $\M^{\J_\ast}$ which is contained in the fiber over an object $A$. Under the equivalence of the previous paragraph, the map $\vphi$ corresponds to a map $\vphi':\F'\lrar \G'$ of functors from $\J$ to $\M_{A//A}$, where $\F'$ and $\G'$ are simply the restrictions of $F$ and $G$ to $\J\subseteq \J_\ast$. It then suffices to show that $\vphi$ is a Reedy (trivial) cofibration in $\M^{\J_\ast}$ if and only $
\vphi'$ is a Reedy (trivial) cofibration in $(\M_{A//A})^{\J}$. 

For an object $i \in \J$, let us denote by $L^{\J}_i: (\M_{A//A})^{\J} \lrar \M_{A//A}\lrar \M$ and $L^{\J_\ast}_i: \M^{\J_\ast} \lrar \M$ the corresponding $i$'th latching object functors, both taking values in $\M$. Our goal is to show that for $i \in \J$, the map
\begin{equation}\label{e:map1}
L^{\J_\ast}_i(\G) \coprod_{L^{\J_\ast}_i(\F)}\F(i) \lrar \G(i) 
\end{equation}
is a (trivial) cofibration if and only if the map 
\begin{equation}\label{e:map2}
L^{\J}_i(\G') \coprod_{L^{\J}_i(\F')}\F'(i) \lrar \G'(i)
\end{equation}
is a (trivial) cofibration in $\M$. 
For an object $i \in \J$ let $\J^{+}_{/i} \subseteq \J_{/i}$ be subcategory whose objects are the non-identity maps $j \lrar i$ in $\J^+$ and whose morphisms are maps in $\J^+$ over $i$, and let $\J^{+}_{\ast/i}$ be the defined similarly. Note that $\J^{+}_{\ast/i}$ is obtained from $\J^{+}_{/i}$ by freely adding an initial object. Consequently, the data of a diagram $\J^{+}_{\ast/i}\lrar \M$ is equivalent (by adjunction) to the data of a diagram $\J^{+}_{/i}\lrar \M_{\F(\ast)/}$. 
It follows that 
$$ L^{\J_\ast}_i(\F) = \displaystyle\mathop{\colim}_{j \rar i \in \J^{+}_{\ast/i}}\F(j) =\left[\displaystyle\mathop{\colim}_{j \rar i \in \J^{+}_{/i}}\F(j)\right]\coprod_{\displaystyle\mathop{\colim}_{j \rar i \in \J^{+}_{/i}}\F(\ast)} \F(\ast) 
=  L^{\J}_i(\F') \coprod_{L^{\J}_i(\F(\ast))} \F(\ast) $$
and similarly
$$ L^{\J_\ast}_i(\G) = L^{\J}_i(\G') \coprod_{L^{\J}_i(\G(\ast))} \G(\ast) $$
where by abuse of notation we considered $\F(\ast)$ and $\G(\ast)$ as constant functors $\J^+_{/i} \lrar \M$ (with value $A$). We now see that both~\ref{e:map1} and~\ref{e:map2} can be identified with the colimit of the diagram
$$ \xymatrix{
\F(i) & L^{\J}_i(\F') \ar[l] \ar^{\Id}[r] & L^{\J}_i(\F') \\
\F(\ast)\ar[u]\ar[d] & L^{\J}_i(\F(\ast)) \ar[r]\ar[l]\ar[d]\ar[u] & L^{\J}_i(\F')\ar[d]\ar_{\Id}[u]  \\
\G(\ast) & L^{\J}_i(\G(\ast)) \ar[l]\ar[r] & L^{\J}_i(\G') \\
}$$
in the category $\M$: for~\ref{e:map1} we first compute the pushouts of the rows and for~\ref{e:map2} we start with the columns, using that $L^{\J}_i$ preserves colimits for the middle column.
\end{proof}

\begin{proof}[Proof of Theorem \ref{t:fib}]
Let $\pi_\text{pre}: \M^{(\NN\times\NN)_\ast}_{\Reedy}\lrar \M$ be the functor evaluating at the basepoint $\ast$. By Lemma \ref{l:relative}, this functor is a relative model category and a (co)Cartesian fibration, whose domain is left proper and combinatorial. 

To see that the functor $\pi: \T\M\lrar \M$ is a relative model category as well, we have to show that $\pi$ is a left Quillen functor for the tangent model structure, by Proposition \ref{p:localization}. For this it suffices to show that its right adjoint $\ast: \M\lrar \T\M$ sends fibrant objects to local objects, i.e.\ parameterized $\Omega$-spectra in $\M$. Indeed, this implies that $\ast$ preserves fibrations between fibrant objects, since fibrations between local objects are just fibrations in $\M^{(\NN\times\NN)_\ast}_{\Reedy}$, so that $\ast$ is right Quillen (\cite[\S 3]{Hir}). But now observe that for any fibrant object $A$, the value $\ast_A$ is simply the constant $(\NN\times\NN)_\ast$-diagram on $A$, which is certainly a Reedy fibrant $\Omega$-spectrum in $\M_{A//A}$.

We conclude that $\pi$ is a relative model category, so that its restriction
$$\xymatrix{
\pi^{\fib}: \T\M\times_\M \M^{\fib}\ar[r] & \M^{\fib}
}$$
to the fibrant objects is a relative model category as well. To see that it is a model fibration, it suffices to verify conditions (a) and (b) of Definition \ref{d:modfib}. For (a), let $f: X\lrar Y$ be a $\pi$-coCartesian map in $\T\M\times_\M \M^{\fib}$ whose image $\pi(f)$ is a weak equivalence in $\M^{\fib}$ and whose domain is cofibrant in the fiber $\T\M_{\pi(X)}$. Then $X$ is cofibrant in the fiber $\big(\M^{(\NN\times\NN)_\ast}_{\Reedy}\big)_{\pi(X)}$ as well, so by Proposition \ref{p:reedygivesmodfib}, the map $X\lrar Y$ is a (Reedy) weak equivalence, hence a stable weak equivalence.

The proof for (b) is exactly the same, using that a fibrant object in a fiber $\T\M_A$ is \emph{in particular} fibrant in $\big(\M^{(\NN\times\NN)_\ast}_{\Reedy}\big)_A$.

Finally, we show that the model fibration $\pi^{\fib}$ is classified by the functor
$$
\M^{\fib}\lrar \ModCat; \hspace{4pt} A\mapsto \Sp(\M_{A//A}).
$$
As we have already seen in Proposition \ref{p:reedygivesmodfib}, the Cartesian and coCartesian fibration underlying $\pi^{\fib}$ is given by
$$
\M^{\fib}\lrar \AdjCat; \hspace{4pt} A\mapsto \big(\M_{A//A}\big)^{\NN\times\NN}.
$$
It remains to show that for each fibrant object $A$, the restriction of the model structure on $\T\M$ to the fiber $\T\M_A$ agrees with $\Sp(\M_{A//A})$, the stable model structure on  $\big(\M_{A//A}\big)^{\NN\times\NN}$. 

Note that $\T\M$ has the same cofibrations and less fibrations than the Reedy model structure on $\M^{(\NN\times \NN)_\ast}$. Consequently, the fibers of $\pi: \T\M\lrar \M$ have the same cofibrations and less fibrations than the fibers of $\pi_\text{pre}: \M^{(\NN\times\NN)_\ast}_{\Reedy}\lrar \M$. In other words, the fiber $\T\M_A$ is a left Bousfield localization of $\pi_\text{pre}^{-1}(A)$, which Proposition \ref{p:reedygivesmodfib} identifies with the Reedy model structure on $(\M_{A//A})^{\NN\times \NN}$.

Both $\T\M_A$ and $\Sp(\M_{A//A})$ are therefore left Bousfield localizations of the Reedy model structure on $(\M_{A//A})^{\NN\times \NN}$, and it suffices to identify their fibrant objects. But by Remark \ref{r:omega}, for any fibrant object $A\in \M$, an object in $\T\M_A$ is fibrant if and only if it is a Reedy fibrant $\Omega$-spectrum in $\M_{A//A}$. These are precisely the fibrant objects in $\Sp(\M_{A//A})$ as well.
\end{proof}

\subsection{Tensor structures on the tangent bundle}\label{s:tanten}
When $\M$ is tensored over a symmetric monoidal (SM for short) model category $\bS$, the category $\M^{(\NN\times\NN)_\ast}$ inherits a natural levelwise tensor structure (see~\cite{Ba07}). In favorable cases, this levelwise tensor structure is compatible with the tangent model structure.
\begin{pro}\label{p:localtensor}
Let $\bS$ be a tractable SM model category, i.e.\ a combinatorial model category with a set $I=\{K_\alpha\lrar L_\alpha\}$ of generating cofibrations with cofibrant domain. Suppose that $\M$ is a model category which is tensored and cotensored over $\bS$ and that $L_S\M$ is a left Bousfield localization of $\M$ at a set of maps $S$ between cofibrant objects. If cotensoring with a cofibrant object in $\bS$ preserves $S$-local objects in $\M$ then $L_S\M$ is tensored and cotensored over $\bS$ as well.
\end{pro}
\begin{proof}
It is enough to check that the pushout-product of a map $i: K_\alpha\lrar L_\alpha$ in $I$ against a trivial cofibration $X\lrar Y$ in $L_S\M$ is a local weak equivalence. If cotensoring with a cofibrant object $K$ in $\bS$ preserves $S$-local objects in $\M$, then the Quillen pair $K\otimes (-): \M \adj \M : (-)^K$ descends to a Quillen pair 
$$
K\otimes (-): L_S\M \adj L_S\M : (-)^K.
$$
Since the objects $K_\alpha$ and $L_\alpha$ are cofibrant, the maps $K_\alpha\otimes X\lrar K_\alpha\otimes Y$ and $L_\alpha\otimes X\lrar L_\alpha\otimes Y$ are trivial cofibrations in $L_S\M$. Since the cobase change of a trivial cofibration is again a trivial cofibration, it follows from the 2-out-of-3 property in $L_S\M$ that the pushout-product map 
$$
K_\alpha\otimes Y\coprod_{K_\alpha\otimes X} L_\alpha\otimes X\lrar L_\alpha\otimes Y
$$
is a weak equivalence in $L_S\M$. 
\end{proof}
\begin{cor}\label{c:tensored}
Let $\M$ be a left proper combinatorial model category which is tensored and cotensored over a tractable SM model category $\bS$. Then $\T\M$ is naturally tensored and cotensored over $\bS$, where the tensor structure is given by
$$ K \otimes \left(B \lrar X_{\bullet\bullet} \lrar B\right) = K \otimes B \lrar K \otimes X_{\bullet\bullet} \lrar K \otimes B $$
and the cotensor is given by
$$ \left(B \lrar X_{\bullet\bullet} \lrar B\right)^K = B^K \lrar (X_{\bullet\bullet})^K \lrar B^K. $$
\end{cor}
\begin{proof}
By~\cite[Lemma 4.2]{Ba07} the levelwise tensor-cotensor structure over $\bS$ is compatible with the Reedy model structure on $\M^{(\NN \times \NN)_\ast}$. To verify the condition of Proposition \ref{p:localtensor}, is suffices to prove that cotensoring with a cofibrant object $K\in\bS$ preserves parameterized $\Om$-spectra. This follows from the fact that cotensoring with $K$ preserves weak equivalences between fibrant objects and homotopy Cartesian squares involving fibrant objects, since $(-)^K: \M \lrar \M$ is right Quillen.
\end{proof}

\begin{exam}
If $\M$ is a simplicial left proper combinatorial model category, then $\T\M$ is naturally a simplicial model category.
\end{exam}

\begin{exam}
If $\M$ is a left proper SM tractable model category, then $\T\M$ is naturally tensored over $\M$.
\end{exam}

\subsection{Comparison with the $\infty$-categorical construction}\label{s:oo-stab}

Recall that any model category $\M$ (and in fact any relative category) has a canonically associated $\infty$-category $\M_\infty$, obtained by formally inverting the weak equivalences of $\M$ (see e.g.~\cite{Hin} for a thorough account, or alternatively, the discussion in~\cite[\S 2.2]{BHH}). Furthermore, a Quillen adjunction $\L :\M\adj \N : \R$ induces an adjunction of $\infty$-categories $\L_\infty : \M_\infty\adj \N_\infty: \R_\infty$ (\cite[Proposition 1.5.1]{Hin}).

Our goal in this section is to show that the construction of (parameterized) spectrum objects described in \S\ref{s:stabilization} and \S\ref{s:tanbun} is a model categorical presentation of its $\infty$-categorical counterpart. Our first step is to show that the $\infty$-category associated to the stabilization $\Sp(\M)$ of a model category $\M$ presents the universal stable $\infty$-category associated to $\M_\infty$, in the sense of \cite[Proposition 1.4.2.22]{Lur14}. 
For this it will be useful to consider the operation of stabilization in the not-necessarily pointed setting. Recall that if $\C$ is a presentable $\infty$-category then the $\infty$-category $\C_\ast \x{\df}{=} \C_{\ast/}$ of objects under the terminal object is the universal pointed presentable $\infty$-category receiving a colimit preserving functor from $\C$. Since any stable $\infty$-category is necessarily pointed we see that any colimit preserving functor from $\C$ to a stable presentable $\infty$-category factors uniquely through $\C_{\ast}$. The composition $\C \lrar \C_{\ast} \lrar \Sp(\C_{\ast})$ thus exhibits $\Sp(\C_{\ast})$ as the universal stable presentable $\infty$-category admitting a colimit preserving functor from $\C$. Given a left proper combinatorial model category $\M$ we will therefore consider $\Sp(\M_\ast)$ also as the stabilization of $\M$, where $\M_{\ast} = \M_{\ast/}$ is equipped with the coslice model structure. We note that when $\M$ is already weakly pointed we have a Quillen equivalence $\M_\ast \x{\simeq}{\adj} \M$ and so this poses no essential ambiguity. We will denote by $\Sig^{\infty}_+: \M \adj \Sp(\M_{\ast}): \Om^{\infty}_+$ the composition of Quillen adjunctions
$$ \xymatrix{
\Sig^{\infty}_+:\M \ar@<1ex>[r]^-{(-) \coprod \ast} & \M_{\ast} \ar@<1ex>[l]^-{\U} \ar@<1ex>[r]^-{\Sig^{\infty}} & \Sp(\M_{\ast}) \ar@<1ex>[l]^-{\Om^{\infty}}: \Om^{\infty}_+ \\
}.$$
We note that the above construction is only appropriate if $\M_{\ast}$ is actually a model for the $\infty$-category $(\M_{\infty})_{\ast}$. We shall begin by addressing this issue.

\begin{lem}\label{l:overmodel}
Let $\M$ be a combinatorial model category and $X\in \M$ an object. Assume either that $X$ is cofibrant or that $\M$ is left proper. Then the natural functor of $\infty$-categories $(\M_{X/})_\infty\lrar (\M_\infty)_{X/}$ is an equivalence.
\end{lem}
\begin{proof}
If $\M$ is left proper then any weak equivalence $f: X\lrar X'$ induces a Quillen equivalence $f_!: \M_{X/} \adj \M_{X'/}: f^*$ and hence an equivalence between the associated $\infty$-categories. Similarly, for any model category the adjunction $f_! \dashv f^*$ is a Quillen equivalence when $f$ is a weak equivalence between cofibrant objects. It therefore suffices to prove the lemma under the assumption that $X$ is fibrant-cofibrant.

Note that for any Quillen equivalence $\L: \N \adj \M: \R$ and a fibrant object $X\in \M$, the induced Quillen pair $\N_{\R(X)/}\adj \M_{X/}$ is a Quillen equivalence as well. 
By the main theorem of~\cite{Dug} there exists a simplicial, left proper combinatorial model category $\M'$, together with a Quillen equivalence $\M'\adj\M$. We may therefore reduce to the case where $\M$ is furthermore simplicial and $X\in \M$ is fibrant-cofibrant, in which case the result follows from \cite[Lemma 6.1.3.13]{Lur09}.
\end{proof}

\begin{pro}\label{p:comparisonwithlurie}
Let $\M$ be a left proper combinatorial model category. Then the functor $(\Om^{\infty}_+)_{\infty}:\Sp(\M_{\ast})_\infty \lrar \M_\infty$ exhibits $\Sp(\M_{\ast})_\infty$ as the stabilization of $\M_\infty$ (in the sense of the universal property of~\cite[Proposition 1.4.2.23]{Lur14}). 
\end{pro}
\begin{proof}
Since $\M$ is left proper, Lemma~\ref{l:overmodel} implies that the natural functor $\big(\M_{\ast}\big)_\infty \lrar (\M_\infty)_{\ast}$ is an equivalence. It therefore suffices to show that for a weakly pointed model category $\M$, the map $(\Om^\infty)_\infty: \Sp(\M)_\infty \lrar \M_\infty$ exhibits $\Sp(\M)_{\infty}$ as the stabilization of the pointed $\infty$-category $\M_\infty$.

Since $\Sp(\M)$ is a left Bousfield localization of $\M^{\NN \times \NN}_{\Reedy}$ (Corollary~\ref{c:existence of Sp}) it follows that the underlying $\infty$-category $\Sp(\M)_{\infty}$ is equivalent to the full subcategory of $(\M^{\NN \times \NN}_{\Reedy})_{\infty}$ spanned by the local objects, i.e., by the $\Om$-spectra. By~\cite[Proposition 4.2.4.4]{Lur09} the natural map
$$ (\M^{\NN \times \NN})_{\infty} \lrar (\M_\infty)^{\NN \times \NN} $$
is an equivalence of $\infty$-categories. We may therefore conclude that $\Sp(\M)_{\infty}$ is equivalent to the full subcategory $\Sp'(\M_{\infty}) \subseteq (\M_\infty)^{\NN \times \NN}$ spanned by those diagrams $\F:\NN \times \NN \lrar \M_{\infty}$ such that $\F(n,m)$ is zero object for $n \neq m$ and $\F$ restricted to each diagonal square is Cartesian.
We now claim that the evaluation at $(0,0)$ functor $\ev_{(0,0)}:\Sp'(\M_{\infty}) \lrar \M_{\infty}$ exhibits $\Sp'(\M_{\infty})$ as the stabilization of $\M_\infty$. By~\cite[Proposition 1.4.2.24]{Lur14} it will suffice to show that $\ev_{(0,0)}$ lifts to an equivalence between $\Sp'(\M_{\infty})$ and the homotopy limit of the tower
\begin{equation}\label{e:tower}
\cdots \lrar \M_\infty \x{\Om}{\lrar} \M_\infty \x{\Om}{\lrar} \M_\infty 
\end{equation}
The proof of this fact is completely analogous to the proof of~\cite[Proposition 8.14]{Lur06}. Indeed, one may consider for each $n$ the $\infty$-category $\D'_n$ of $(\NN_{\leq n}\times \NN_{\leq n})$-diagrams in $\M_\infty$ which are contractible off-diagonal and have Cartesian squares on the diagonal. It follows from Lemma 8.12 and Lemma 8.13 of \cite{Lur06} (as well as \cite[Proposition 4.3.2.15]{Lur09}) that the functor $\ev_{(n, n)}: \D'_n\lrar \M_\infty$ is a trivial Kan fibration (hence a categorical equivalence). Under these equivalences, the restriction functor $\D'_{n+1}\lrar \D'_n$ is identified with the loop functor $\Om: \M_\infty\lrar \M_\infty$. It follows that the homotopy limit of the tower~\ref{e:tower}
can be identified with the homotopy limit of the tower of restriction functors $\{\cdots \lrar \D'_{2}\lrar \D'_1\lrar \D'_0\}$. Since these restriction functors are categorical fibrations between $\infty$-categories, the homotopy limit agrees with the actual limit, which is the $\infty$-category $\Sp'(\M_\infty)$.
\end{proof}

\begin{cor}\label{c:BFspectrainstable}
If $\M$ is a stable model category, then the adjunction $\Sigma^\infty \dashv \Omega^\infty$ of Corollary \ref{c:existence of Sp} is a Quillen equivalence.
\end{cor}

\begin{rem}\label{r:pre-spectra}
If $\M$ is a (weakly pointed, combinatorial) model category which is not left proper we can still consider the full relative subcategory $\Sp'(\M)\subseteq \M^{\NN\times \NN}$ spanned by $\Om$-spectra (with weak equivalences the levelwise weak equivalences). The composite functor $\Sp'(\M)_\infty\lrar (\M^{\NN\times \NN})_\infty\x{\sim}{\lrar} (\M_\infty)^{\NN\times\NN}$ identifies $\Sp'(\M)_\infty$ with the full sub-$\infty$-category $\Sp'(\M_{\infty}) \subseteq (\M_\infty)^{\NN \times \NN}$ spanned by those diagrams which are contractible off diagonal and have Cartesian diagonal squares. The proof of Proposition~\ref{p:comparisonwithlurie} now implies that for \textbf{any} weakly pointed combinatorial model category $\M$, the stabilization of $\M_\infty$ can be modeled by the relative category $\Sp'(\M)$.
\end{rem}

\begin{rem}\label{r:spectra}
Corollary~\ref{p:comparisonwithlurie} can be used to compare $\Sp(\M)$ with other models for the stabilizations appearing in the literature. For example, the construction of Hovery (\cite{Hov}) using Bousfield-Friedlander spectra is also known to present the $\infty$-categorical stabilization (see~\cite[Proposition 4.15]{Rob12}). Since both models for the stabilization are combinatorial model categories they must consequently be related by a chain of Quillen equivalences (see~\cite[Remark A.3.7.7]{Lur09}). 
Another closely related model is that of \textbf{reduced excisive functors} (see e.g.~\cite{Lyd98}). Let $\SS^{\fin}_*$ denote the relative category of pointed finite simplicial sets. When $\M$ is left proper and combinatorial we may form the left Bousfield localization $\Exc_\ast(\M)$ of the projective model structure on $\M^{\SS^{\fin}_{\ast}}$ in which the local objects are the relative reduced excisive functors. Restriction along $\iota: \{S^0\} \hrar  \SS^{\fin}_{\ast}$ then yields a right Quillen functor $\iota^*: \Exc_\ast(\M) \lrar \M$ and by~\cite[Proposition 4.2.4.4]{Lur09} and~\cite[\S 1.4.2]{Lur14} the induced functor $\iota^*_\infty: (\Exc_\ast(\M))_{\infty} \lrar \M_{\infty}$ 
exhibits $(\Exc_\ast(\M))_{\infty}$ as the stabilization of $\M_\infty$.  
In this case one can even construct a direct right Quillen equivalence $\Exc_\ast(\M) \lrar \Sp(\M)$ by restricting along a suspension spectrum object $f: \NN \times \NN \lrar \SS^{\fin}_{\ast}$ with $f(0,0) \cong S^0$. 
\end{rem}

The above results show that for any fibrant-cofibrant object $A$ of a left proper combinatorial model category $\M$, the stable model category $\Sp(\M_{A//A})$ is a model for the $\infty$-categorical stabilization of $(\M_\infty)_{A//A}$. This shows that $\T_A(\M_\infty)$ is equivalence to $(\T_A\M)_\infty$. Our final goal in this section is to compare the model-categorical tangent bundle of $\M$ to the $\infty$-categorical tangent bundle of $\M_\infty$:
\begin{thm}
Let $\M$ be a left proper combinatorial model category. The induced map of $\infty$-categories $\pi_\infty: \T\M_\infty \lrar \M_\infty$ exhibits $\T\M_\infty$ as a tangent bundle to $\M_\infty$. 
\end{thm}
\begin{proof}
Let $j: [1]\lrar (\NN \times \NN)_\ast$ be the inclusion of the arrow $(0, 0)\lrar \ast$ in $(\NN \times \NN)_\ast$. Restriction along $j$ induces a diagram of right Quillen functors
$$\xymatrix@R=1pc{
\T\M\ar[rr]^-{j^*}\ar[rd]_{\pi} & & \M^{[1]}_{\Reedy}\ar[ld]^{\ev_1}\\
& \M. & 
}$$
which induces a triangle of $\infty$-categories 
$$\xymatrix@R=1pc{
(\T\M)_{\infty}\ar[rr]^-{j^*}\ar[rd]_{\pi} & & (\M_{\infty})^{[1]} \ar[ld]^{\ev_1}\\
& \M_{\infty}. & 
}$$

To see that this triangle exhibits $(\T\M)_\infty$ as the tangent bundle to $\M_\infty$, let $\T\M'\subseteq \T\M$ be the full relative subcategory on objects in $\T\M$ whose image in $\M$ is fibrant and let $\M'^{[1]}\subseteq \M^{[1]}$ be the full subcategory of fibrations with fibrant codomain. Both of these inclusions are equivalences of relative categories, with homotopy inverse given by a fibrant replacement functor. It will hence suffice to show that for every fibrant $A \in \M$ the induced map $((\T\M)_\infty)_A \lrar (\M_{\infty})^{[1]}_A \simeq (\M_\infty)_{/A}$ exhibits $((\T\M)_\infty)_A$ as the stabilization of $(\M_\infty)_{/A}$. But this now follows directly from Theorem~\ref{t:fib}, \cite[Proposition 2.1.4]{Hin} and the fiberwise comparison given by Proposition~\ref{p:comparisonwithlurie}.

\end{proof}

\end{document}